\documentclass[a4paper, 12pt]{article}
\usepackage{amsthm}
\usepackage{amsmath}
\usepackage{amssymb}
\usepackage{enumerate}
\usepackage{latexsym}
\usepackage{verbatim}
\usepackage[mathscr]{eucal}
\usepackage{bm}
\usepackage{mathrsfs}
\usepackage{indentfirst}
\usepackage{typearea}
\typearea{16}
\allowdisplaybreaks[1]

\newcommand{\relmid}[1]{\mathrel{}\middle#1\mathrel{}}

\begin{document}
\title{On the calculation of the ramified Siegel series}  
\date{}
\author{Masahiro Watanabe\thanks{Kyoto University, Kitashirakawa Oiwake-Cho, Sakyo-ku, Kyoto-shi, Kyoto-fu, 606-8224, Japan, m.watanabe@math.kyoto-u.ac.jp\quad ORCID: 0000-0003-1707-333X}}
\maketitle

\theoremstyle{plain}
\newtheorem{theorem}{Theorem}[section]
\newtheorem{lemma}{Lemma}[section]
\newtheorem{proposition}{Proposition}[section]
\theoremstyle{definition}
\newtheorem{definition}{Definition}[section]
\newtheorem{conjecture}{Conjecture}[section]
\theoremstyle{remark}
\newtheorem{remark}{{\bf Remark}}[section]
\newtheorem{corollary}{\bf {Corollary}}[section]
\newtheorem{formula}{\bf {Formula}}[section]

\begin{abstract}
	The ramified Siegel series is an important factor that appears in the Fourier coefficient of the Siegel Eisenstein series.
	Many formulas for the ramified Siegel series under various conditions are already known.
	However, an explicit formula for the general case has not yet been obtained.
	We derive a formula for the Siegel series with arbitrary dimension $n$, assuming that the additive character $\psi$ is primitive.
	Our results cover nonarchimedean, non-dyadic local fields $F$, including the case $F=\mathbb{Q}_p$.
	We also give explicit values of the ramified Siegel series for degrees $n=1, 2,$ and $3$.
\end{abstract}

\section{Introduction}

Let $k\ge2$ be an integer.
We recall the Fourier expansion of the Eisenstein series.
For $\tau\in\mathbb{H}:=\{z\in\mathbb{C} \mid \mathrm{Im}(z)>0\}$,
let $G_{2k}(\tau)$ be the Eisenstein series of weight $2k$, defined by
$$G_{2k}(\tau):=\sum_{(m, n)\in\mathbb{Z}^2\backslash\{(0,0)\}}\frac{1}{(m+n\tau)^{2k}}.$$
The Fourier expansion of $G_{2k}(\tau)$ is well known as follows.
$$G_{2k}(\tau)=2\zeta(2k)+\frac{2(2\pi i)^{2k}}{(2k-1)!}\sum_{n=1}^\infty \sigma_{2k-1}(n)q^n.$$
Here $\zeta(z)$ is Riemann's zeta function, and $\sigma_p(n)$ is the divisor sum function.
We write $q=\exp(2\pi i\tau)$.
Hence, the number-theoretic function $\sigma_{2k-1}(n)$ appears in the Fourier coefficients.

We now consider the Siegel Eisenstein series. Let $k\ge2$ and $l, n\ge1$ be integers, and $\psi$ be a Dirichlet character modulo $l$.
The Siegel Eisenstein series $E_{k,l, \psi}^n(Z)$ is defined as
$$E_{k,l, \psi}^n(Z):= \sum_{\begin{pmatrix}
		A & B \\ C & D
	\end{pmatrix} \in \Gamma_\infty^n\backslash\Gamma_0^n(l)}\psi(\det D)\det (CZ+D)^{-k},$$
where the sets $\Gamma_\infty^n$ and $\Gamma_0^n(l)$ are subgroups of $\Gamma^n=\mathrm{Sp}_n(\mathbb{Z})$
defined by
\begin{align*}
	\Gamma_\infty^n &:= \left\{\begin{pmatrix}
		A & B \\ C & D
	\end{pmatrix} \in \Gamma^n \relmid| C=0
	\right\}, \\
	\Gamma_0^n(l) &:=\left\{\begin{pmatrix}
		A & B \\ C & D
	\end{pmatrix} \in \Gamma^n \relmid| C\equiv0 \pmod{l}
	\right\},
\end{align*}
respectively.
Here $Z\in\mathbb{H}^n$ denotes $\{Z\in M_n(\mathbb{C}) \mid {}^tZ=Z, \,\, \mathrm{Im}(Z)\,\, \text{is positive definite}\}$.

The Fourier expansion of the Siegel Eisenstein series is given by
$$E_{k,l, \psi}^n(Z)=\sum_{A\in S_n^*, A>0}c(A)\exp(2\pi i\, \mathrm{Tr}(AZ)),$$
where
$$S_n^*
:=\left\{A=(a_{ij})\in\mathrm{Sym}_n(\mathbb{Q}) \mid a_{ii}\in\mathbb{Z}, \,\, a_{ij}\in\tfrac{1}{2}\mathbb{Z}\,\, (i\neq j)\right\},$$
and $c(A)$ is a constant depending only on the matrix $A$. The symbol $A>0$ means that the matrix $A$ is positive definite.

The ramified Siegel series is an important factor that appears in $c(A)$ and is partially calculated in many papers.
In this article, we calculate the ramified Siegel series under some conditions.

More precisely, the Fourier coefficient $c(A)$ is calculated as
\begin{align*}
	c(A) = \frac{2^{-\frac{n(n-1)}{{2}}}(-2\pi i)^{nk}}{\pi^{\frac{n(n-1)}{4}}\prod_{j=0}^{m-1}\Gamma(s-j/2)}
	(\det A)^{k-\frac{n+1}{2}}\prod_{p:\,\, \text{prime}}b_n^p(A, s),
\end{align*}
when $A\ge0$ and the factor $b_n^p(A, s)$ is called the ramified Siegel series when $p$ divides $l$.
Here we denote $\Gamma(s)$ the Gamma function.

Many results of the formula of the Siegel series are already known.
Katsurada \cite{Katsurada} gave the explicit formula for the case $l=1$ (full level). 
When $n=2$, the results are well known when the Dirichlet character $\psi$ is primitive, that is, $\psi=\prod\psi_p$.
For example, Mizuno \cite{Mizuno} treats the square-free level case, while Takemori \cite{Takemori} treats arbitrary level $l$.
Gunji calculated when $n=2$, in \cite{Gunj}, and when $n=3$, in \cite{Gunji}, respectively.
In this article, we calculate the ramified Siegel series for arbitrary $n$.

In \cite{Gunji2}, Gunji proved a result equivalent to Theorem 1.1.
We extend this result when $F$ is a non-dyadic local field, while all results above are $F=\mathbb{Q}_p$ case.
By calculating the integration of the Siegel series, we divide the domain of integration by the orbits of the action of $\Gamma_0$ and use the method of Sato and Hironaka \cite{SatoHironaka}.
We define a Weil constant $\alpha_\psi(x)$ ($x\in F^\times$) so that the integrals $I(a), I^*(a)$ in \cite{SatoHironaka}, which are well-known as Gauss sums when $F=\mathbb{Q}_p$, can be evaluated uniformly over $F$ using this constant.

The ramified Siegel series is also related to the degenerate Whittaker function \cite{Moeglin}.
When we use the functional equation of the Whittaker function calculated in \cite{Sweet} and \cite{FEIkeda}, we obtain the functional equation of the Siegel series.
The calculation of the functional equation is in progress, and a detailed account will be given elsewhere.

We state our main results.

Let $G=\mathrm{Sp}_n(F)$ be the symplectic group of rank $n$ over a nonarchimedean, non-dyadic field $F$.
Define $\mathfrak{o}$, $\mathfrak{p}$ as the ring of integers of $F$ and the maximal ideal of $\mathfrak{o}$, respectively.
We fix a prime element $\pi\in\mathfrak{p}$.
We write $K=\mathrm{Sp}_n(\mathfrak{o})$.
Let $P$ be the standard Siegel parabolic subgroup of $G$, and the subgroup $\Gamma$ of $G$ is defined by
$$\Gamma=\left\{
\begin{pmatrix}
	A & B \\ C & D
\end{pmatrix}
\in K\relmid| C\equiv0\bmod\mathfrak{p}\right\}.$$
We define $\{w_i\}_{0\le i\le n}$ by
$$w_i=\left(\begin{array}{rr|rr}
	1_{n-i} &&& \\
	&&& -1_i \\ \hline
	&& 1_{n-i} & \\
	& 1_i &&
\end{array}\right).$$
The set $\{w_i\}_{0\le i\le n}$ is a complete set of representatives of the double coset $P\backslash G/\Gamma$.

Let $\psi$ be an additive character of $F$ of order $0$ and $\omega$ be a character of $F^\times$ satisfying $\omega^2=1$.
We define $I_n(\omega, s)=\mathrm{Ind}_P^G(\omega\circ|\det|^s)$  the space of smooth functions on $G$ satisfying
$$f\left(
\begin{pmatrix}
	A & * \\ 0 & {}^t\!A^{-1}
\end{pmatrix}
g\right)=\omega(\det A)|\det A|^{s+\frac{n+1}{2}}f(g).$$
We also define the space
$$I_n\left(\omega, s-\frac{n+1}{2}\right)^{\Gamma, \omega}=
\left\{f\in I_n\left(\omega, s-\frac{n+1}{2}\right)\relmid| f(gk)=\omega(k)f(g)\,\, \text{for any}\,\, k\in\Gamma\right\}.$$
This space is spanned by the functions $\{f_i\}_{0\le i\le n}$ which satisfy $f_i(w_j)=\delta_{ij}$.

A ramified Siegel series associated with
$\displaystyle\varphi\in I_n\left(\omega, s-\frac{n+1}{2}\right)^{\Gamma, \omega}$ is defined by an integral
$$\int_{\mathrm{Sym}_n(F)}\varphi\left(w_n
\begin{pmatrix}
	1 & X \\ 0 & 1
\end{pmatrix}
\right)\psi(-\mathrm{tr}(BX))dX.$$
We write $S_t(B, s)^\omega$ for the ramified Siegel series associated with $\varphi=f_t$ ($0\le t\le n$). 

We divide the domain of integration $\mathrm{Sym}_n(F)$ by the orbits of the action of
$\Gamma_0=\{\gamma=(\gamma_{ij})\in\mathrm{GL}_n(\mathfrak{o})\mid\gamma_{ij}\in\mathfrak{p}\,(i>j)\}$
and use a formula for the volume of each orbit due to Sato and Hironaka \cite{SatoHironaka}.

We need some notations.
We define $\chi$, a ramified nontrivial character of $F^\times$, by $\chi(x)=\langle \pi, x\rangle$.
Here we note $\langle\,, \,\rangle$ a Hilbert symbol.
Put $I=\{1, 2, \cdots, n\}$ and consider the standard action of $\mathfrak{S}_n$ on $I$.
Let $\sigma \in \mathfrak{S}_n$ satisfy $\sigma^2 = 1$, and let
$I = I_0 \cup \cdots \cup I_r$ be a partition of $I$ into disjoint $\sigma$-stable subsets.
We define the numbers $c_1^{(k)}(\sigma), c_2(\sigma), n_k, n^{(k)}, n(k)$, associated with the partition $\{I_i\}$, by
\begin{align*}
	c_1^{(k)}(\sigma)&=\#\{i\in I_k \mid \sigma(i)=i\}, \\
	c_2(\sigma)&=\frac{1}{2}\#\{i\in I\mid \sigma(i)\neq i\}, 
\end{align*}
and
\begin{align*}
	n_k&=\#I_k,\quad n^{(k)}=\sum_{l=k}^rn_l,\quad n(k)=\frac{n^{(k)}(n^{(k)}+1)}{2}.
\end{align*}
We also put
\begin{align*}
	\tau(\{I_i\})&=\sum_{l=1}^r\#\{(i, j)\in I_l\times (I_0\cup\cdots\cup I_{l-1})\mid j<i\}, \\
	t(\sigma, \{I_i\})&=\sum_{l=0}^r\#\{(i, j)\in I_l\times I_l \mid i<j<\sigma(i),\, \sigma(j)<\sigma(i)\}, \\
	e_{\sigma, i, u}&=\begin{cases}
		0 & (u\le i,\, u\le\sigma(i)) \\
		1 & (\sigma(i)<u\le i\,\,\, \text{or}\,\,\, i<u\le\sigma(i)) \\
		2 & (i<u,\, \sigma(i)<u).
	\end{cases}
\end{align*}
Let $B$ a matrix $B=\mathrm{diag}(v_1\pi^{e_1}, \cdots, v_n\pi^{e_n})$, where $v_i\in\mathfrak{o}^\times$ and $0\le e_1\le \cdots\le e_n$. 
Set
\begin{align*}
	b_l(\sigma, B)&=\min\left\{
	\{e_i \mid i\in I_l, \sigma(i)>i\} \cup \{e_i+1 \mid i\in I_l, \sigma(i)\le i\}\right\}, \\
	B_i(\lambda)
	&=\{ k \mid 1\le k\le i-1,\,\, e_k+\lambda<0,\,\, e_k\not\equiv\lambda\bmod2\} \\
	&\quad\cup\, \{ k \mid i+1\le k\le n,\,\, e_k+\lambda+2<0,\,\, e_k\not\equiv\lambda\bmod2\}, \\
	\tilde{\rho}_{l, \lambda}(\sigma; B)
	&=\frac{1}{2}\sum_{i\in I_l}\sum_{u=1}^n
	\min\{e_u+e_{\sigma, i ,u}+\lambda,\, 0\}.
\end{align*}
We further define
\begin{align*}
	\xi_{i, \lambda}(B)_\chi
	&=\prod_{k \in B_i(\lambda)}\chi(v_k)
	\times
	\begin{cases}
		0 & e_i+\lambda\ge0,\,\, \#B_i(\lambda) : \text{even} \\
		(1-q^{-1})\chi(-1)^{[\#B_i(\lambda)/2]+1} & e_i+\lambda\ge0,\,\, \#B_i(\lambda) : \text{odd} \\
		\chi(v_i)\chi(-1)^{[\#B_i(\lambda)/2]+1} & e_i+\lambda=-1,\,\, \#B_i(\lambda) : \text{even} \\
		-q^{-1/2}\chi(-1)^{[\#B_i(\lambda)/2]+1} & e_i+\lambda=-1,\,\, \#B_i(\lambda) : \text{odd}. 
	\end{cases}
\end{align*}
For $\varphi\in\{f_0, f_1 \cdots, f_n\}$, we give an explicit formula of the ramified Siegel series as follows.

\begin{theorem}
	The ramified Siegel series $S_t(B, s)^\chi$ $(0\le t\le n)$ is given by
	\begin{align*}
		S_t(B, s)^\chi
		&=\alpha_\psi(\pi)^{n-t}\sum_{\substack{\sigma\in\mathfrak{S}_n \\ \sigma^2=1}}
		(1-q^{-1})^{c_2(\sigma)}q^{-c_2(\sigma)}
		\sum_{\substack{I=I_0\cup\cdots\cup I_r \\ n^{(k)}=t}}q^{-\tau(\{I_i\})-t(\sigma, \{I_i\})}\frac{
			(1-q^{-1})^{\sum_{l=k}^rc_1^{(l)}(\sigma)}q^{n(k)}}
		{\prod_{l=k}^r(q^{n(l)}-1)} \\
		&\quad\times\sum_{\{\nu\}_k^t}\prod_{l=0}^{k-1}\chi(-1)^{\nu_l(n^{(l)}-n^{(k)})}
		q^{\nu_l((sn^{(l)}-n(l))-(sn^{(k)}-n(k)))+\tilde{\rho}_{l, \nu_0+\cdots+\nu_l}(\sigma; B)}
		\prod_{\substack{i\in I_l \\ \sigma(i)=i}}\xi_{i, \nu_0+\cdots+\nu_l}(B)_\chi.
	\end{align*}
	Here the summation with respect to $\{\nu\}_k^t$ for $k\ge1$ is taken over the finite set
	$$\left\{(\nu_0, \nu_1, \cdots, \nu_{k-1})\in \mathbb{Z}\times\mathbb{Z}_{>0}^{k-1} \relmid|
	-b_l(\sigma, B)\le\nu_0+\nu_1+\cdots+\nu_l\le-1\,\, (0\le l\le k-1)\right\}.$$
	Moreover, in $\displaystyle \sum_{\substack{I=I_0\cup\cdots\cup I_r\\ n^{(k)}=t}}$
	we sum over all $\sigma$-stable partitions $I=I_0\cup\cdots\cup I_r$; 
	for each such partition, the index $k$ with $0\le k\le r$ and $n^{(k)}=t$ is uniquely determined.
\end{theorem}

In Section 2, we set up notation and terminology, and review some of the standard facts on the induced representation.
In Section 3, we explain how to calculate the ramified Siegel series.
More precisely, Section 3.1 sets up the orbit-decomposition framework;
Section 3.2 extends the Sato-Hironaka orbit and volume formulae to any nonarchimedean, non-dyadic $F$;
Section 3.3 applies them to write $S_t(B, s)^\omega$ as a finite $\Gamma_0$-orbit sum;
Section 3.4 introduces $\alpha_\psi(\pi)$ to evaluate the resulting Gauss integrals uniformly over $F$. 
In Section 4, we give an explicit formula for general $\varphi=f_t$.
Finally in Section 5, we give explicit values for $n=1, 2, 3$.

The author wishes to thank Keiichi Gunji of Chiba Institute of Technology for sending us his article \cite{Gunji2} when it is a preprint.
The author would like to gratefully thank Tamotsu Ikeda of Kyoto University for his constant support.

\section{The ramified Siegel series}

Let $F$ be a nonarchimedean local field of characteristic zero. 
We denote the ring of integers of $F$ by $\mathfrak{o}=\mathfrak{o}_F$. 
We write $\mathfrak{p}=\mathfrak{p}_F$, $\mathfrak{k}=\mathfrak{k}_F=\mathfrak{o}_F/\mathfrak{p}_F$ 
for the maximal ideal and the residue field of $\mathfrak{o}$, respectively. 
We fix a prime element $\pi\in\mathfrak{p}$.
Let $q$ denote the cardinality of $\mathfrak{k}$.
We will assume later that $q$ is odd.

The set of symmetric matrices of degree $n$ over $F$ is denoted by $\mathrm{Sym}_n(F)$,
and the subset of non-degenerate symmetric matrices is denoted by $S_n(F)$.

The symplectic group of degree $n$ over $F$ is defined by
$$G=\mathrm{Sp}_n(F):=
\{M\in \mathrm{GL}_{2n}(F)\mid{}^t\!Mw_nM=w_n\}$$
where we write ${}^t\!M$ for the transpose of a matrix $M$, and $w_n$ for
$\displaystyle
\begin{pmatrix}
	0 & -1_n \\ 1_n & 0
\end{pmatrix}$.
Let $K=\mathrm{Sp}_n(\mathfrak{o})$ denote a maximal compact subgroup of $G$.
Define
$$
P=\left\{
\begin{pmatrix}
	A & B \\ C & D
\end{pmatrix}
\in \mathrm{Sp}_n(F)\relmid| C=0 \right\},$$
the Siegel parabolic subgroup of $G$.
The decomposition $G=PK$ is called the Iwasawa decomposition.
Set
$$\Gamma=\left\{
\begin{pmatrix}
	A & B \\ C & D
\end{pmatrix}
\in K\relmid| C\equiv0\bmod\mathfrak{p}\right\}.$$

Let $\omega$ be a character of $F^\times$ satisfying $\omega^2=1$.
Later we only consider $\omega=\mathbf{1}$ the trivial character or $\omega=\chi$ which satisfies $\chi(x) = \langle \pi, x\rangle$; a ramified nontrivial character.
We define a character $\omega^\Gamma$ on $\Gamma$ as
$$\omega^\Gamma\left(
\begin{pmatrix}
	A & B \\ C & D
\end{pmatrix}
\right)=\omega(\det D)$$
and we write $\omega$ for $\omega^\Gamma$ by abuse of notation.

Later we only consider when B is non-degenerate, that is, $\det B\neq0$.

Define $I_n(\omega, s)=\mathrm{Ind}_P^G(\omega\circ|\det|^s)$ as the space of the induced representation,
the space of smooth functions on $G$ satisfying
$$f\left(
\begin{pmatrix}
	A & * \\ 0 & {}^t\!A^{-1}
\end{pmatrix}
g\right)=\omega(\det A)|\det A|^{s+\frac{n+1}{2}}f(g),$$
and we also define $I_n(\omega, s)^{\Gamma, \omega}$ as
$$I_n(\omega, s)^{\Gamma, \omega}=\{f\in I(\omega, s)\mid f(gk)=\omega(k)f(g)\,\, \text{for all}\,\, k\in\Gamma\}.$$

The double coset
$P\backslash G/\Gamma\simeq(P\cap K)\backslash K/\Gamma=\Gamma\backslash K/\Gamma$
is a finite set and one can choose a complete set of representatives $\{w_i\}_{0\le i\le n}$ by
$$w_i=\left(\begin{array}{rr|rr}
	1_{n-i} &&& \\
	&&& -1_i \\ \hline
	&& 1_{n-i} & \\
	& 1_i &&
\end{array}\right).$$
It is known that
$\displaystyle
\begin{pmatrix}
	A & B \\ C & D
\end{pmatrix}
\in K$ is an element of $\Gamma w_i\Gamma$ if and only if 
the matrix $C$, considered modulo $\mathfrak{p}$, has rank $i$.

Each $f\in I_n(\omega, s)^{\Gamma, \omega}$ is determined by its value on $K$, and hence
is determined by $\{f(w_i)\}_{0\le i\le n}$.
We give the following proposition.
\begin{proposition}
	Define $f_i\in I(\omega, s)^{\Gamma ,\omega}$ satisfying $f_i(w_j)=\delta_{ij}$.
	Then the $\mathbb{C}$-vector space $I(\omega, s)^{\Gamma ,\omega}$ is spanned by the functions $\{f_i\}_{0\le i\le n}$, in other words, 
	$$I_n(\omega, s)^{\Gamma, \omega}=\bigoplus_{i=0}^n\mathbb{C}f_i.$$
\end{proposition}

Now we define the Siegel series.
\begin{definition}
	The Siegel series is an integral
	$$\int_{\mathrm{Sym}_n(F)}\varphi\left(w_n
	\begin{pmatrix}
		1 & X \\ 0 & 1
	\end{pmatrix}
	\right)\psi(-\mathrm{tr}(BX))dX,$$
	where $\psi$ is an additive character of $F$, $B\in\mathrm{Sym}_n(F)$
	and $\varphi\in I_n\left(\omega, s-\frac{n+1}{2}\right)^{\Gamma, \omega}$.
	$S_t(B, s)^\omega$ is denoted by the Siegel series when $\varphi=f_t$ ($0\le t\le n$).
\end{definition}
\begin{remark}
	When $\mathrm{Re}\,s\gg0$, the Siegel series is absolutely integrable.
	This integral easily extends to a meromorphic function on $\mathbb{C}$.
\end{remark}
\begin{remark}
	At the beginning we introduced the (ramified) Siegel series by a local integral, 
	whereas in what follows we work with the series $S_t(B, s)^\omega$.
	By Shimura \cite[\S 18.10]{Shimura}, the Euler $p$-factor $b_p^{(n)}(A, s)$ occurring in the Fourier 
	coefficient of the Siegel Eisenstein series coincides with $S_t(B, s)^\omega$ up to an explicit constant . 
	Hence the computation of $b_p^{(n)}(A, s)$ reduces to that of $S_t(B, s)^\omega$, and we will compute $S_t(B, s)^\omega$ below.
\end{remark}

\section{Calculation of the Siegel series associated with $\varphi=f_t$}
\subsection{Preparations for the orbit decomposition}

In this section, we consider the case where $\varphi=f_t$ ($0\le t\le n$). 
Let $B\in \mathrm{Sym}_n(F)$.
Recall that we define the Siegel series as
\begin{align*}
	S_t(B, s)^\chi
	&=\int_{\mathrm{Sym}_n(F)}f_t\left(w_n
	\begin{pmatrix}
		1 & X \\ 0 & 1
	\end{pmatrix}\right)
	\psi(-\mathrm{tr}(BX))dX \\
	&=\int_{\mathrm{Sym}_n(F)}f_t\left(
	\begin{pmatrix}
		0 & -1 \\ 1 & X
	\end{pmatrix}\right)
	\psi(-\mathrm{tr}(BX))dX.
\end{align*}
Note that the function $f_t$ is taken in the space $I_n\left(\omega, s-\frac{n+1}{2}\right)^{\Gamma, \omega}$.
We may assume that  $X$ is an invertible matrix since the measure of $\mathrm{Sym}_n(F)\backslash S_n(F)$ is zero.
We need the Iwasawa decomposition of the matrix
$\begin{pmatrix} 0 & -1 \\ 1 & X \end{pmatrix}$.

The following lemma is well known and is called the Jordan splitting:
\begin{lemma} \label{Jordan}
	For $X\in\mathrm{Sym}_n(\mathfrak{o})\cap S_n(F)$, there are $U\in\mathrm{GL}_n(\mathfrak{o})$ and
	diagonal matrix $Y$ such that $X={}^tUYU$.
	Moreover, when we write $Y=\mathrm{diag}(\alpha_1\pi^{p_1}, \cdots, \alpha_n\pi^{p_n})$
	$(\alpha_i\in \mathfrak{o}^\times, p_i\ge0)$ then
	for each $m\ge0$, the value
	$$\#\{i\mid p_i=m\}\quad\text{and}\quad\prod_{p_i=m} \alpha_i$$
	are uniquely determined by the matrix $X$, modulo ${\mathfrak{o}^\times}^2$.
\end{lemma}

When we write $X={}^tUYU$ where $U\in\mathrm{GL}_n(\mathfrak{o})$ we have

$$
\begin{pmatrix}
	0 & -1 \\ 1 & {}^tUYU
\end{pmatrix}
=
\begin{pmatrix}
	U^{-1} & \\ & {}^tU
\end{pmatrix}
\begin{pmatrix}
	0 & -1 \\ 1 & Y
\end{pmatrix}
\begin{pmatrix}
	{}^tU^{-1} & \\ & U
\end{pmatrix}
$$
Then
\begin{align*}
	f_t\left(
	\begin{pmatrix}
		0 & -1 \\ 1 & X
	\end{pmatrix}
	\right)
	&=f_t\left(
	\begin{pmatrix}
		U^{-1} & \\ & {}^tU
	\end{pmatrix}
	\begin{pmatrix}
		0 & -1 \\ 1 & Y
	\end{pmatrix}
	\begin{pmatrix}
		{}^tU^{-1} & \\ & U
	\end{pmatrix}
	\right) \\
	&=\omega(\det U^{-1}{}^tU^{-1})\left|\det U^{-1}{}^tU^{-1}\right|^{-s}
	f_t\left(
	\begin{pmatrix}
		0 & -1 \\ 1 & Y
	\end{pmatrix}
	\right) \\
	&=f_t\left(
	\begin{pmatrix}
		0 & -1 \\ 1 & Y
	\end{pmatrix}
	\right)
\end{align*}
since $\omega^2=1$ and $\det U\in\mathfrak{o}^\times$.

We write
$Y=\begin{pmatrix}
	Y_1 & \\ & Y_2
\end{pmatrix}$
where $Y_1$ and $Y_2$ are diagonal matrices of degree $r$ and $n-r$ ($0\le r\le n$), respectively.
We assume that $Y_1\in M_r(\mathfrak{o})$ and $Y_2^{-1}\in\mathfrak{p}M_{n-r}(\mathfrak{o})$. Since
$$
\left(\begin{array}{rr|rr}
	0 && -1 & \\
	& 0 && -1 \\ \hline
	1 && Y_1 & \\
	& 1 && Y_2
\end{array}\right)
=
\left(\begin{array}{rr|rr}
	0 && -1 & \\
	& 0 && -1 \\ \hline
	1 && 0 & \\
	& 1 && Y_2
\end{array}\right)
\left(\begin{array}{rr|rr}
	1 && Y_1 & \\
	& 1 && 0 \\ \hline
	0 && 1 & \\
	& 0 && 1
\end{array}\right)
$$
so we may assume that $Y_1=0$.
Now we consider an Iwasawa decomposition;
$$
\left(\begin{array}{rr|rr}
	0 && -1 & \\
	& 0 && -1 \\ \hline
	1 && 0 & \\
	& 1 && Y_2
\end{array}\right)
=
\left(\begin{array}{rr|rr}
	1 && 0 & \\
	& 1 && -Y_2^{-1} \\ \hline
	0 && 1 & \\
	& 0 && 1
\end{array}\right)
\left(\begin{array}{rr|rr}
	1 && 0 & \\
	& Y_2^{-1} && 0 \\ \hline
	0 && 1 & \\
	& 0 && Y_2
\end{array}\right)
\left(\begin{array}{rr|rr}
	0 && -1 & \\
	& 1 && 0 \\ \hline
	1 && 0 & \\
	& Y_2^{-1} && 1
\end{array}\right).
$$
We note that ${}^tY_2=Y_2$, since $Y_2$ is diagonal.
Therefore in order to $f_t\left(
\begin{pmatrix}
	0 & -1 \\ 1 & Y
\end{pmatrix}\right)\neq0$,
we assume the rank of the matrix
$\begin{pmatrix}
	1 & \\ & Y_2^{-1}
\end{pmatrix}$
as mod $\mathfrak{p}$ $\left(=
\begin{pmatrix}
	1 & \\ & 0
\end{pmatrix}
\right)$ is $t$.
We define $S_{n, t}(F)$ the subset of $S_n(F)$ by
$$S_{n, t}(F)
=\{ X\in S_n(F)\mid (*): \#\{i \mid e_i<0\}=n-t\},$$
where the notation $e_i$ is the same as in Lemma 3.1.
Assume from now on that $X\in S_{n,t}(F)$.
From the discussion above, we have
\begin{align*}
	f_t\left(
	\begin{pmatrix}
		0 & -1 \\ 1 & X
	\end{pmatrix}\right)
	&=f_t\left(\left(
	\begin{array}{rr|rr}
		1 && 0 & \\
		& 1 && -Y_2^{-1} \\ \hline
		0 && 1 & \\
		& 0 && 1
	\end{array}\right)
	\left(\begin{array}{rr|rr}
		1 && 0 & \\
		& Y_2^{-1} && 0 \\ \hline
		0 && 1 & \\
		& 0 && Y_2
	\end{array}\right)
	\left(\begin{array}{rr|rr}
		0 && -1 & \\
		& 1 && 0 \\ \hline
		1 && 0 & \\
		& Y_2^{-1} && 1
	\end{array}\right)\right) \\
	&=\omega(\det Y_2^{-1})|\det Y_2^{-1}|^s
	f_t\left(\left(
	\begin{array}{rr|rr}
		0 && -1 & \\
		& 1 && 0 \\ \hline
		1 && 0 & \\
		& Y_2^{-1} && 1
	\end{array}\right)\right) \\
	&=\omega(\det Y_2)|\det Y_2|^{-s},
\end{align*}
so the ramified Siegel series can be written as
\begin{align*}
	S_t(B, s)^\omega
	&=\int_{S_{n, t}(F)}f_t\left(
	\begin{pmatrix}
		0 & -1 \\ 1 & X
	\end{pmatrix}\right)
	\psi(-\mathrm{tr}(BX))dX \\
	&=\int_{S_{n, t}(F)}
	\omega\left(\det Y_2\right)\left|\det Y_2\right|^{-s}\psi(-\mathrm{tr}(BX))dX.
\end{align*}

\subsection{Theorems of Sato and Hironaka}

For the proofs of the following results when $F=\mathbb{Q}_p$ we refer the reader to \cite{SatoHironaka}.
The same proof remains valid when $F$ is an arbitrary nonarchimedean, non-dyadic local field, using standard properties of a nonarchimedean local field. 

We put
$$S_n(F)=\{X\in\mathrm{Sym}_n(F)\mid \det X\neq0\},$$
and we define
$$\Gamma_0=\{\gamma=(\gamma_{ij})\in\mathrm{GL}_n(\mathfrak{o})\mid
\gamma_{ij}\in\mathfrak{p}\, (i>j)\}$$
which acts on $S_n(F)$ by $Y\mapsto \gamma\cdot Y=\gamma Y{}^t\gamma$.
We fix a non-square unit
$\delta\in\mathfrak{o}^\times\backslash(\mathfrak{o}^\times)^2$.
The following theorem determines the orbits of the action.
Put $I=\{1, 2, \cdots, n\}$ and consider the standard action of $\mathfrak{S}_n$ on $I$.

Later the notation $h_i$ is different from that in the original paper (written as $e_i$).

\begin{theorem}
	[\cite{SatoHironaka} Theorem 2.1]\label{SH1}
	Let $\Lambda_n$ be the collection of
	$(\sigma, h, \varepsilon)\in\mathfrak{S}_n\times\mathbb{Z}^n\times\{1, \delta\}^n$ satisfying
	$$\sigma^2=1,\,\, h_{\sigma(i)}=h_i\,\,(i\in I),\,\, \varepsilon_i=1\,(i\in I, \sigma(i)\neq i).$$
	For a $(\sigma, h, \varepsilon)\in\Lambda_n$, we define a symmetric matrix $S_{\sigma, h, \varepsilon}$ by
	$$S_{\sigma, h, \varepsilon}=(s_{ij}),\,\, s_{ij}=\varepsilon_i\pi^{h_i}\delta_{i, \sigma(j)},$$
	where $\delta_{i, \sigma(j)}$ is the Kronecker delta.
	Then the set $\{S_{\sigma, h, \varepsilon}\mid (\sigma, h, \varepsilon)\in\Lambda_n\}$ gives
	the complete set of representatives of $\Gamma_0$-equivalence classes in $S_n(F)$.
\end{theorem}

Before writing the second theorem, we need some preparation.

For $Y\in S_n(F)$, we define
$$\alpha(\Gamma_0; Y)=\lim_{l\to\infty}q^{-ln(n-1)/2}N_l(\Gamma_0; Y),$$
where
$$N_l(\Gamma_0; Y)=\#\{\gamma\in\Gamma_0\bmod \mathfrak{p}^l
\mid \gamma Y{}^t\gamma\equiv Y\bmod \mathfrak{p}^l\}.$$

We normalize the Haar measures $d\gamma$ on $M_n(F)$ and $dY$ on $\mathrm{Sym}_n(F)$,
respectively, by
$$\int_{M_n(\mathfrak{o})}d\gamma=1,\quad \int_{\mathrm{Sym}_n(\mathfrak{o})}dY=1.$$
\begin{theorem}
	[\cite{SatoHironaka} Proposition 1.2]\label{SH2}
	Let $Y_0\in S_n(F)$ then the following integral formula holds for any continuous function
	$f$ on $\Gamma_0\cdot Y_0$.
	$$\int_{\Gamma_0\cdot Y_0}f(Y)dY=\alpha(\Gamma_0; Y_0)^{-1}\int_{\Gamma_0}f(\gamma Y_0{}^t\gamma)d\gamma.$$
\end{theorem}

Now we introduce some notation.
For a $(\sigma, h, \varepsilon)\in\Lambda_n$, we define the integers $\lambda_0, \lambda_1, \cdots, \lambda_r$ by
$$\{h_1, \cdots, h_n\}=\{\lambda_0, \lambda_1, \cdots, \lambda_r\}
\quad \text{with} \quad \lambda_0<\lambda_1<\cdots<\lambda_r.$$
We also put
$$I_i=\{j\in I\mid h_j=\lambda_i\}\quad (0\le i\le r).$$
Then $I_0, \cdots, I_r$ are disjoint $\sigma$-stable subsets of $I$ and $I=I_0\cup I_1\cup\cdots\cup I_r$.
We also put
$$I^{(i)}=I_i\cup I_{i+1}\cup\cdots\cup I_r\quad (0\le i\le r).$$
We set
$$n_i=\#(I_i),\,\, n^{(i)}=\#(I^{(i)})=n_i+\cdots+n_r,\,\, n(i)=\frac{n^{(i)}(n^{(i)}+1)}{2}.$$
Put
$$\nu_i=\lambda_i-\lambda_{i-1}\,\, (1\le i\le r),\quad \nu_0=\lambda_0.$$
Then $\nu_0\in\mathbb{Z}$ and $\nu_1, \cdots, \nu_r\in\mathbb{Z}_{>0}$.
\begin{theorem}
	[\cite{SatoHironaka} Theorem 2.2]\label{SH3}
	Put
	\begin{align*}
		c_1(\sigma)&=\#\{i\in I\mid \sigma(i)=i\}, \\
		c_2(\sigma)&=\frac{1}{2}\#\{i\in I\mid \sigma(i)\neq i\}, \\
		t(\sigma, \{I_i\})&=
		\sum_{l=0}^r\#\{(i, j)\in I_l\times I_l\mid i<j<\sigma(i),\,\, \sigma(j)<\sigma(i)\}, \\
		\tau(\{I_i\})&=
		\sum_{l=1}^r\#\{(i, j)\in I_l\times (I_0\cup\cdots\cup I_{l-1})\mid j<i\}.
	\end{align*}
	Then we have
	$$\alpha(\Gamma_0; S_{\sigma, h, \varepsilon})
	=2^{c_1(\sigma)}(1-q^{-1})^{c_2(\sigma)}q^{c(\sigma, h, \varepsilon)},$$
	where
	$$c(\sigma, h, \varepsilon)
	=-\frac{n(n-1)}{2}+\tau(\{I_i\})+t(\sigma, \{I_i\})+c_2(\sigma)+\sum_{l=0}^r\nu_ln(l).$$
\end{theorem}

For $a\in F$, we put
$$I(a)=\int_{\mathfrak{o}}\psi(ax^2)dx,\,\, I^*(a)=\int_{\mathfrak{o}^\times}\psi(ax^2)dx=I(a)-\frac{1}{q}I(a\pi^2).$$
For $T, Y\in S_n(F)$, put
$$\mathscr{G}_{\Gamma_0}(Y, T)
=\int_{\Gamma_0}\psi(-\mathrm{tr}(Y\cdot T[\gamma]))d\gamma.$$
Here we write $T[\gamma]={}^t\gamma T\gamma$.

\begin{theorem}
	[\cite{SatoHironaka} Proposition 3.3]\label{SH4}
	Let $T=\mathrm{diag}(v_1\pi^{\beta_1}, v_2\pi^{\beta_2}, \cdots, v_n\pi^{\beta_n})$ ($v_i\in\mathfrak{o}^\times,\,
	\beta_i\in\mathbb{Z}$).
	For $(\sigma, h, \varepsilon)\in\Lambda_n$, the character sum $\mathscr{G}_{\Gamma_0}(S_{\sigma, h, \varepsilon}, T)$
	vanishes unless
	\begin{align*}
	h_i\ge \begin{cases}
		-\beta_i-1 &\mathrm{if}\,\, \sigma(i)\le i \\
		-\beta_i &\mathrm{if}\,\, \sigma(i)>i
	\end{cases} \tag{T}
	\end{align*}
	for any $i\in I$.
	When the condition above is satisfied, we have
	\begin{align*}
		\mathscr{G}_{\Gamma_0}(S_{\sigma, h, \varepsilon}, T)
		=&(1-q^{-1})^{2c_2(\sigma)}q^{-\frac{n(n-1)}{2}+d(\sigma, h, \beta)} \\
		&\times\prod_{\substack{i=1 \\ \sigma(i)=i}}^n
		\left\{I^*(-\varepsilon_iv_i\pi^{h_i+\beta_i})\prod_{k=1}^{i-1}I(-\varepsilon_iv_k\pi^{h_i+\beta_k})
		\prod_{k=i+1}^nI(-\varepsilon_iv_k\pi^{h_i+\beta_k+2})\right\},
	\end{align*}
	where
	$$
	d(\sigma, h, \beta)
	=\sum_{\substack{i=1 \\ \sigma(i)>i}}^n \left\{
	\sum_{k=1}^{i-1}\min\{h_i+\beta_k, 0\}+\sum_{k=i+1}^{\sigma(i)-1}\min\{h_i+\beta_k+1, 0\}
	+\sum_{k=\sigma(i)+1}^n\min\{h_i+\beta_k+2, 0\}\right\}.
	$$
\end{theorem}

\subsection{Calculation of the integral}
We now compute the integral
\begin{align*}
	S_t(B, s)^\omega=\int_{S_{n, t}(F)}
	\omega\left(\det Y_2\right)\left|\det Y_2\right|^{-s}\psi(-\mathrm{tr}(BX))dX
	\quad (0\le t\le n)
\end{align*}
by using the above theorems.

First, we divide the domain of integration by $\Gamma_0$ orbits.
Each representatives $S_{\sigma, h, \varepsilon}$ is in the element of $S_{n, t}(F)$
if and only if
\begin{align*}
	(*): \#\{i \mid h_i<0\}=n-t.
\end{align*}
Therefore
$$S_t(B, s)^\omega
=\sum_{\sigma, h, \varepsilon (*)}
\int_{\Gamma_0S_{\sigma, h, \varepsilon}}\omega(\det Y_2)|\det Y_2|^{-s}
\psi(-\mathrm{tr}(BX))dX.$$
\begin{lemma}
	The term $\omega(\det Y_2)|\det Y_2|^{-s}$ depends only on its $\Gamma_0$-orbit.
\end{lemma}
\begin{proof}
	Let $X'$ be a matrix with the same equivalence class as $X$, so we can write $X' = \gamma X{}^t\gamma$ for some $\gamma\in \Gamma_0$.
	We note that $\Gamma_0\subset\mathrm{GL}_n(\mathfrak{o})$.
	Therefore the two matrices $X$ and $X'$ are $\mathrm{GL}_n(\mathfrak{o})$-equivalent.
	
	Here, strictly speaking, from Lemma \ref{Jordan}, $\displaystyle \det Y_2=\prod_{v_i<0}\alpha_i$ is determined only modulo ${\mathfrak{o}^\times}^2$. 
	However, since $\omega$ is quadratic (i.e. $\omega^2=1$), we have $\omega(u^2)=1$ for all $u\in\mathfrak{o}^\times$, and moreover $|u^2|=1$. 
	Hence both $\omega(\det Y_2)$ and $|\det Y_2|$ are unchanged when $\det Y_2$ is multiplied by a square unit, so $\omega(\det Y_2)\,|\det Y_2|^{-s}$ depends only on the $\mathrm{GL}_n(\mathfrak{o})$-equivalence class, in particular, only on the $\Gamma_0$-orbit of $X$.
\end{proof}

Let $S_{\sigma, h, \varepsilon}$ satisfy the condition $(*)$.
Let
$$\{\mu_1, \mu_2, \cdots, \mu_{n-t}\}, \,\, \mu_1<\mu_2<\cdots<\mu_{n-t}$$
be the set of integers $\mu: 1\le \mu\le n$ such that $h_\mu<0$.
We define a square matrix $S_{\sigma, h, \varepsilon}^{(t)}$ of order $n-t$ defined by
\begin{align*}
	\left(S_{\sigma, h, \varepsilon}^{(t)}\right)_{ij}
	=\left(S_{\sigma, h, \varepsilon}\right)_{\mu_i, \,\mu_j}.
\end{align*}
Then we note that $Y_2$ and $S_{\sigma, h, \varepsilon}^{(t)}$ are $\mathrm{GL}_{n-t}(\mathfrak{o})$-equivalent.

We denote $c_{\sigma, h, \varepsilon,\omega}(s)
=\omega(\det S_{\sigma, h, \varepsilon}^{(t)})|\det S_{\sigma, h, \varepsilon}^{(t)}|^{-s}$.
Using the previous lemma,
$$S_t(B, s)^\omega
=\sum_{\sigma, h, \varepsilon (*)}c_{\sigma, h, \varepsilon,\omega}(s)
\int_{\Gamma_0S_{\sigma, h, \varepsilon}}\psi(-\mathrm{tr}(BX))dX.$$

Now using Theorem \ref{SH2} when $f(X)=\psi(-\mathrm{tr}(BX))$, we can deduce
\begin{align*}
	\int_{\Gamma_0S_{\sigma, h, \varepsilon}}\psi(-\mathrm{tr}(BX))dX
	&=\frac{1}{\alpha(\Gamma_0; S_{\sigma, h, \varepsilon})}
	\int_{\Gamma_0}\psi(-\mathrm{tr}(B\gamma S_{\sigma, h, \varepsilon}{}^t\gamma))d\gamma \\
	&=\frac{1}{\alpha(\Gamma_0; S_{\sigma, h, \varepsilon})}
	\int_{\Gamma_0}\psi(-\mathrm{tr}(S_{\sigma, h, \varepsilon} \cdot B[\gamma]))d\gamma \\
	&=\frac{\mathscr{G}_{\Gamma_0}(S_{\sigma, h, \varepsilon}, B)}{\alpha(\Gamma_0; S_{\sigma, h, \varepsilon})}.
\end{align*}
We can write the main theorem of this article:
\begin{theorem}\label{main}
	Let the notation is as above. Then we have
	$$S_t(B, s)^\omega=\sum_{\sigma, h, \varepsilon (*)}c_{\sigma, h, \varepsilon,\omega}(s)
	\frac{\mathscr{G}_{\Gamma_0}(S_{\sigma, h, \varepsilon}, B)}{\alpha(\Gamma_0; S_{\sigma, h, \varepsilon})}.$$
\end{theorem}

Here we note that the sum for $\sigma, h, \varepsilon$ is finite.

\subsection{Weil constant}
We recall the definition of the Weil constant as in \cite{FEIkeda}.

For each Schwartz function $\phi\in\mathcal{S}(F)$,
the Fourier transform $\hat{\phi}$ is defined by
$$
\hat{\phi}(x)=\int_F \phi(y)\psi(xy)dy.$$
Note that the Haar measure $dy$ satisfying $\displaystyle\int_{\mathfrak{o}}dy=1$ is the 
self-dual Haar measure for the Fourier transform $\phi\mapsto\hat{\phi}$.

\begin{definition}
	Let $\psi$ be an additive character over $F$ of order $0$, i.e., $\psi|_{\mathfrak{o}}=1$, and let $a\in F^\times$.
	The Weil constant $\alpha_\psi(a)$ is a complex number satisfying
	$$\int_F\phi(x)\psi(ax^2)dx=
	\alpha_\psi(a)|2a|^{-\frac{1}{2}}\int_F\hat{\phi}(x)\psi\left(-\frac{x^2}{4a}\right)dx \eqno(1)$$
	for any $\phi\in\mathcal{S}(F)$.
\end{definition}

The following lemmas are fundamental.
Let $\langle *, *\rangle$ denote the Hilbert symbol of index 2 on $F$.
\begin{lemma} \label{Hilbert}
	For any $a, b\in F^\times$, 
	$$\frac{\alpha_\psi(a)\alpha_\psi(b)}{\alpha_\psi(ab)\alpha_\psi(1)}=\langle a, b\rangle.$$
\end{lemma}

Here we note that, for $a, b\in F^\times$, $\alpha_\psi(ab^2)=\alpha_\psi(a)$.
The next lemma states that the Gauss integral, $\displaystyle I(a)=\int_{\mathfrak{o}}\psi(ax^2)dx$, can be calculated by using the Weil constant, for a nonarchimedean, non-dyadic local field $F$.

\begin{lemma}
	Let $a\in F^\times$. If $\mathrm{ord}(a)<0$, then $I(a)=\alpha_\psi(a)|a|^{-\frac{1}{2}}$.
\end{lemma}
\begin{proof}
	We assume $\phi=1_{\mathfrak{o}}$, which is the characteristic function of $\mathfrak{o}$.
	$\hat{\phi}=1_{\mathfrak{o}}$ is well known, so (1) gives
	$$\int_{\mathfrak{o}}\psi(ax^2)dx
	=\alpha_\psi(a)|2a|^{-\frac{1}{2}}\int_{\mathfrak{o}}\psi\left(-\frac{x^2}{4a}\right)dx.$$
	The integral $\displaystyle\int_{\mathfrak{o}}\psi\left(-\frac{x^2}{4a}\right)dx$ is 1 because
	$\mathrm{ord}\left(\displaystyle -\frac{x^2}{4a}\right)\ge0$.
	Since $F$ is non-dyadic, we have $|2|=1$.
	This completes the proof.
\end{proof}

\begin{lemma} \label{Weildelta}
	For $a\in F^\times$,
	$$\alpha_\psi(\delta a)=\begin{cases}
		\alpha_\psi(a) & (\mathrm{ord}\,a:\mathrm{even}) \\
		-\alpha_\psi(a) & (\mathrm{ord}\,a:\mathrm{odd}).
	\end{cases}$$
\end{lemma}
\begin{proof}
	It is sufficient to show the following equations.
	\begin{align}
		\alpha_\psi(\delta)&=\alpha_\psi(1)=1. \tag{2}\\
		\alpha_\psi(\delta\pi)&=-\alpha_\psi(\pi). \tag{3}
	\end{align}
	
	(2) is trivial. To prove (3), we use Lemma \ref{Hilbert}. We have
	$$\frac{\alpha_\psi(\delta)\alpha_\psi(\pi)}{\alpha_\psi(\delta\pi)\alpha_\psi(1)}
	=\langle \delta, \pi \rangle.$$
	Because the fact that $\alpha_\psi(\delta)=\alpha_\psi(1)=1$ and $\langle \delta, \pi \rangle=-1$, the lemma is proved.
\end{proof}

\begin{remark}
	When $F$ is the $p$-adic number field $\mathbb{Q}_p$ ($p$: odd prime)
	and additive character $\psi$ is given by $\psi(x)=\exp(-2\pi x\sqrt{-1})$ ($x\in\mathbb{Z}[1/p]$),
	the value of $\alpha_\psi(p)$ is
	$$
	\alpha_\psi(p)=\begin{cases}
		1 & p\equiv1 \pmod{4} \\
		\sqrt{-1} & p\equiv 3 \pmod{4}.
	\end{cases}
	$$
\end{remark}

\section{An explicit formula for the Siegel series}
In this section, we obtain an explicit formula for
$$S_t(B, s)^\omega=\sum_{\sigma, h, \varepsilon (*)}c_{\sigma, h, \varepsilon, \omega}(s)
\frac{\mathscr{G}_\Gamma(S_{\sigma, h, \varepsilon}, B)}{\alpha(\Gamma; S_{\sigma, h, \varepsilon})}$$
with the notation in the previous section.
Also we put $B=\mathrm{diag}(v_1\pi^{e_1}, \cdots, v_n\pi^{e_n})$, where
$v_i\in\mathfrak{o}^\times$ and $0\le e_1\le e_2\le \cdots\le e_n$.
From lemma \ref{Jordan}, we assume $B$ as this type without loss of generality.

First, we assume that $\omega=\chi$.

We put
\begin{align*}
	e_{\sigma, i, u}&=\begin{cases}
		0 & (u\le i,\, u\le\sigma(i)) \\
		1 & (\sigma(i)<u\le i\,\,\, \text{or}\,\,\, i<u\le\sigma(i)) \\
		2 & (i<u,\, \sigma(i)<u),
	\end{cases} \\
	b_l(\sigma, B)&=\min\{
	\{e_i \mid i\in I_l, \sigma(i)>i\} \cup
	\{e_i+1 \mid i\in I_l, \sigma(i)\le i\}\}, \\
	B_i(\lambda)
	&=\{ k \mid 1\le k\le i-1,\,\, e_k+\lambda<0,\,\, e_k\not\equiv\lambda\bmod2\} \\
	&\quad\cup\, \{ k \mid i+1\le k\le n,\,\, e_k+\lambda+2<0,\,\, e_k\not\equiv\lambda\bmod2\}, \\
	\tilde{\rho}_{l, \lambda}(\sigma; B)
	&=\frac{1}{2}\sum_{i\in I_l}\sum_{u=1}^n
	\min\{e_u+e_{\sigma, i ,u}+\lambda,\, 0\}, \\
	\xi_{i, \lambda}(B)_\chi
	&=\prod_{k \in B_i(\lambda)}\chi(v_k)
	\times
	\begin{cases}
		0 & e_i+\lambda\ge0,\,\, \#B_i(\lambda) : \text{even} \\
		(1-q^{-1})\chi(-1)^{[\#B_i(\lambda)/2]+1} & e_i+\lambda\ge0,\,\, \#B_i(\lambda) : \text{odd} \\
		\chi(v_i)\chi(-1)^{[\#B_i(\lambda)/2]+1} & e_i+\lambda=-1,\,\, \#B_i(\lambda) : \text{even} \\
		-q^{-1/2}\chi(-1)^{[\#B_i(\lambda)/2]+1} & e_i+\lambda=-1,\,\, \#B_i(\lambda) : \text{odd}. 
	\end{cases}
\end{align*}
We now state the following theorem.

\begin{theorem}\label{mainthm}
	The Siegel series $S_t(B, s)^\chi$ associated with the function $f_t$ $(0\le t\le n)$ is given by:
	\begin{align*}
		S_t(B, s)^\chi
		&=\alpha_\psi(\pi)^{n-t}\sum_{\substack{\sigma\in\mathfrak{S}_n \\ \sigma^2=1}}
		(1-q^{-1})^{c_2(\sigma)}q^{-c_2(\sigma)}
		\sum_{\substack{I=I_0\cup\cdots\cup I_r \\ n^{(k)}=t}}q^{-\tau(\{I_i\})-t(\sigma, \{I_i\})}\frac{
			(1-q^{-1})^{\sum_{l=k}^rc_1^{(l)}(\sigma)}q^{n(k)}}
		{\prod_{l=k}^r(q^{n(l)}-1)} \\
		&\quad\times\sum_{\{\nu\}_k^t}\prod_{l=0}^{k-1}\chi(-1)^{\nu_l(n^{(l)}-n^{(k)})}
		q^{\nu_l((sn^{(l)}-n(l))-(sn^{(k)}-n(k)))+\tilde{\rho}_{l, \nu_0+\cdots+\nu_l}(\sigma; B)}
		\prod_{\substack{i\in I_l \\ \sigma(i)=i}}\xi_{i, \nu_0+\cdots+\nu_l}(B)_\chi.
	\end{align*}
	
	Here the summation with respect to $\{\nu\}_k^t$ for $k\ge1$ is taken over the finite set
	$$\left\{(\nu_0, \nu_1, \cdots, \nu_{k-1})\in \mathbb{Z}\times\mathbb{Z}_{>0}^{k-1} \relmid|
	-b_l(\sigma, B)\le\nu_0+\nu_1+\cdots+\nu_l\le-1\, (0\le l\le k-1)\right\}$$
	and we put
	$$c_1^{(l)}(\sigma)=\#\{i\in I_l \mid \sigma(i)=i\}.$$
	Moreover, in $\displaystyle \sum_{\substack{I=I_0\cup\cdots\cup I_r\\ n^{(k)}=t}}$
	we sum over all $\sigma$-stable partitions $I=I_0\cup\cdots\cup I_r$; 
	for each such partition, the index $k$ with $0\le k\le r$ and $n^{(k)}=t$ is uniquely determined.
\end{theorem}

\begin{proof}
	From Theorem \ref{main}, we consider
	$$S_t(B, s)^\chi=\sum_{\sigma, h, \varepsilon (*)}c_{\sigma, h, \varepsilon, \chi}(s)
	\frac{\mathscr{G}_{\Gamma_0}(S_{\sigma, h, \varepsilon}, B)}{\alpha(\Gamma_0; S_{\sigma, h, \varepsilon})},$$
	where
	\begin{align*}
		c_{\sigma, h, \varepsilon, \chi}(s)
		&=\chi(\det S_{\sigma, h, \varepsilon}^{(t)})|\det S_{\sigma, h, \varepsilon}^{(t)}|^{-s} \\
		&=\chi\left((-1)^{c_2^t(\sigma)}\prod_{h_i<0}\varepsilon_i\pi^{\sum_{h_j<0}h_j}\right)
		\left|(-1)^{c_2^t(\sigma)}\prod_{h_i<0}\varepsilon_i\pi^{\sum_{h_j<0}h_j}\right|^{-s} \\
		&=\chi(-1)^{c_2^t(\sigma)}\chi(\pi)^{\sum_{h_j<0} h_j}\prod_{h_i<0}\chi(\varepsilon_i)q^{s\sum_{h_j<0} h_j}
	\end{align*}
	since $\chi(\pi)=\chi(-1)$. Here we note that $\chi$ satisfies $\chi(-\pi)=\langle\pi, -\pi\rangle=1$, and we put $$c_2^t(\sigma)=\frac{1}{2}\#\{i\in I \mid \sigma(i)\neq i,\, h_i<0\}.$$
	
	So we have
	\begin{align*}
		S_t(B, s)^\chi
		&=\int_{S_{n, t}(F)}f_t\left(
		\begin{pmatrix}
			0 & -1 \\ 1 & X
		\end{pmatrix}\right)
		\psi(-\mathrm{tr}(BX))dX \\
		&=\sum_{\sigma, h, \varepsilon (*)}
		\int_{\Gamma_0S_{\sigma, h, \varepsilon}}
		\chi\left(\det Y_2^{-1}\right)\left|\det Y_2^{-1}\right|^{-s}\psi(-\mathrm{tr}(BX))dX \\
		&=\sum_{\sigma, h, \varepsilon (*)}
		\chi(-1)^{c_2^t(\sigma)}\chi(\pi)^{\sum_{h_j<0} h_j}\prod_{e_i<0}\chi(\varepsilon_i)q^{s\sum_{h_j<0} h_j}
		\frac{\mathscr{G}_{\Gamma_0}(B, S_{\sigma, h, \varepsilon})}{\alpha(\Gamma_0; S_{\sigma, h, \varepsilon})}.
	\end{align*}
	
	By using Theorem \ref{SH3} and Theorem \ref{SH4}, we have
	\begin{align*}
		S_t(B, s)^\chi&=\sum_{\sigma, h, \varepsilon (*)}
		\chi(-1)^{c_2^t(\sigma)}\chi(-1)^{\sum_{h_j<0} h_j}\prod_{h_i<0}\chi(\varepsilon_i)q^{s\sum_{h_j<0} h_j}
		(1-q^{-1})^{2c_2(\sigma)}q^{-\frac{n(n-1)}{2}+d(\sigma, h, e)} \\
		&\quad\times\prod_{\substack{i=1 \\ \sigma(i)=i}}^n
		\left\{I^*(-\varepsilon_iv_i\pi^{h_i+e_i})\prod_{k=1}^{i-1}I(-\varepsilon_iv_k\pi^{h_i+e_k})
		\prod_{k=i+1}^nI(-\varepsilon_iv_k\pi^{h_i+e_k+2})\right\} \\
		&\quad\times 2^{-c_1(\sigma)}(1-q^{-1})^{-c_2(\sigma)}
		q^{\frac{n(n-1)}{2}-\tau(\{I_i\})-t(\sigma, \{I_i\})-c_2(\sigma)-\sum_{l=0}^r\nu_ln(l)},
	\end{align*}
	where the summation with respect to
	$(\sigma, h, \varepsilon)$ is taken over all $(\sigma, h, \varepsilon)\in\Lambda_n$
	satisfying the conditions $(*)$ and (T) in Theorem \ref{SH4}.
	
	We replace the sum with respect to $(\sigma, h)$ with $(\sigma, I, \{\nu\})$.
	In order to do so, we rewrite the terms $\sum_{h_j<0}h_j$ and $d(\sigma, h, e)$.
	The term $\sum_{h_j<0}h_j$ is
	\begin{align*}
		\sum_{h_j<0}h_j
		&=n_0\nu_0+n_1(\nu_0+\nu_1)+\cdots+n_{k-1}(\nu_0+\nu_1+\cdots+\nu_{k-1}) \\
		&=\sum_{l=0}^{k-1}\nu_l(n_l+\cdots+n_{k-1}) \\
		&=\sum_{l=0}^{k-1}\nu_l(n^{(l)}-n^{(k)}).
	\end{align*}
	Also the term $d(\sigma, h, e)$ is
	\begin{align*}
		d(\sigma, h, e)
		&=\sum_{\substack{i=1 \\ \sigma(i)>i}}^n \Biggl\{
		\sum_{u=1}^{i-1}\min\{h_i+e_u, 0\}+\sum_{u=i+1}^{\sigma(i)-1}\min\{h_i+e_u+1, 0\}
		+\sum_{u=\sigma(i)+1}^n\min\{h_i+e_u+2, 0\}\Biggl\} \\
		&=\sum_{\substack{i=1 \\ \sigma(i)>i}}^n\sum_{u\neq i, \sigma(i)}
		\min\{h_i+e_u+e_{\sigma, i, u}, 0\} \\
		&=\sum_{l=k}^r\sum_{u\neq i, \sigma(i)}
		\min\{\lambda_l+e_u+e_{\sigma, i, u}, 0\}=: \tilde{d}(\sigma, \lambda, e).
	\end{align*}
	Then it follows that
	\begin{align*}
		S_t(B, s)^\chi&=\sum_{\substack{\sigma\in\mathfrak{S}_n \\ \sigma^2=1}}
		\chi(-1)^{c_2^t(\sigma)}2^{-c_1(\sigma)}(1-q^{-1})^{c_2(\sigma)}q^{-c_2(\sigma)}
		\sum_{I=I_0\cup\cdots\cup I_r}q^{-\tau(\{I_i\})-t(\sigma, \{I_i\})} \\
		&\quad\times\sum_{\nu_0, \cdots, \nu_r(**)}\chi(-1)^{\sum_{l=0}^{k-1}\nu_l(n^{(l)}-n^{(k)})}
		q^{\tilde{d}(\sigma, \lambda, e)+\sum_{l=0}^{k-1}\nu_l(sn^{(l)}-sn^{(k)}-n(l))} \\
		&\quad\times\sum_\varepsilon\chi(\varepsilon_i)'\prod_{l=0}^r\prod_{\substack{i\in I_l \\ \sigma(i)=i}}
		\left\{I^*(-\varepsilon_iv_i\pi^{\lambda_l+e_i})\prod_{k=1}^{i-1}I(-\varepsilon_iv_k\pi^{\lambda_l+e_k})
		\prod_{k=i+1}^nI(-\varepsilon_iv_k\pi^{\lambda_l+e_k+2})\right\},
	\end{align*}
	where $\lambda_l=\nu_0+\nu_1+\cdots+\nu_l$\,\,($l=0, 1, \cdots, r$).
	The summation with respect to
	$\nu_0, \nu_1, \cdots, \nu_r$ is taken over all $(\nu_0, \nu_1, \cdots, \nu_r)$ satisfying
	$$
	(**) \quad \nu_1, \cdots, \nu_r\ge1,\,\,\, -b_l(\sigma, B)\le\lambda_l\le-1 \quad (l=0, 1, \cdots, r).$$
	Here we also define $\chi(\varepsilon_i)'$ as
	$$\chi(\varepsilon_i)'=
	\begin{cases}
		\chi(\varepsilon_i) & (\lambda_i<0) \\
		1 & (\lambda_i\ge0).
	\end{cases}$$
	Put
	$$Q_{l, \lambda_l}(\sigma; B)
	=\prod_{\substack{i\in I_l \\ \sigma(i)=i}}q_{i, \lambda_l}(\sigma; B),$$
	where
	$$q_{i, \lambda_l}(\sigma; B)
	=\sum_{\varepsilon=1, \delta}\left\{\chi(\varepsilon)
	I^*(-\varepsilon v_i\pi^{e_i+\lambda_l})
	\prod_{k=1}^{i-1}I(-\varepsilon v_k\pi^{e_k+\lambda_l})
	\prod_{k=i+1}^nI(-\varepsilon v_k\pi^{e_k+\lambda_l+2})\right\}.$$
	We note that, when $\sigma(i)\neq i$, $\varepsilon_i$ must be $1$, so
	$\prod_i\chi(\varepsilon_i)=\prod_{\sigma(i)=i}\chi(\varepsilon_i).$
	
	We conclude
	\begin{align*}
		S_t(B, s)^\chi&=\sum_{\substack{\sigma\in\mathfrak{S}_n \\ \sigma^2=1}}
		\chi(-1)^{c_2^t(\sigma)}2^{-c_1(\sigma)}(1-q^{-1})^{c_2(\sigma)}q^{-c_2(\sigma)}
		\sum_{I=I_0\cup\cdots\cup I_r}q^{-\tau(\{I_i\})-t(\sigma, \{I_i\})} \\
		&\quad\times\sum_{\nu_0, \cdots, \nu_r (**)'}\chi(-1)^{\sum_{l=0}^{k-1}\nu_l(n^{(l)}-n^{(k)})}
		q^{\tilde{d}(\sigma, \lambda, e)+\sum_{l=0}^{k-1}\nu_l(sn^{(l)}-sn^{(k)}-n(l))}
		\prod_{l=0}^r\prod_{\substack{i\in I_l \\ \sigma(i)=i}}q'_{i, \lambda_l}(\sigma; B) \tag{4}
	\end{align*}
	where $\{\nu_i\}$ takes
	$$ (**)' \quad \nu_1, \cdots, \nu_r\ge1,\,\,\, -b_l(\sigma, B)\le\lambda_l,\quad
	\sum_{\substack{l \\ \lambda_l<0}}n_l=n-t.$$
	We define
	$$q'_{i, \lambda_l}(\sigma; B)
	=\sum_{\varepsilon=1, \delta}\left\{\chi(\varepsilon)'
	I^*(-\varepsilon v_i\pi^{e_i+\lambda_l})
	\prod_{k=1}^{i-1}I(-\varepsilon v_k\pi^{e_k+\lambda_l})\prod_{k=i+1}^nI(-\varepsilon v_k\pi^{e_k+\lambda_l+2})\right\}$$
	and since $i\in I_l$ we get  
	$$q'_{i, \lambda_l}(\sigma; B)=
	\begin{cases}
		q_{i, \lambda_l}(\sigma; B) & (\lambda_l<0) \\
		2(1-q^{-1}) & (\lambda_l\ge0).
	\end{cases}$$
	Here we recall that, for $v\in\mathfrak{o}^\times$ and $k\in\mathbb{Z}$,
	$$I(v\pi^k)=
	\begin{cases}
		1 & k\ge0 \\
		\alpha_\psi(v\pi^k)q^{k/2} & k<0,
	\end{cases} \quad
	I^*(v\pi^k)=I(v\pi^k)-\frac{1}{q}I(v\pi^{k+2}).
	$$

	We divide the summation with respect to $(\nu_0, \cdots, \nu_r)$ as follows:
	$$
	\sum_{\nu_0, \cdots, \nu_r (**)'}
	=\sum_{k=0}^{r+1}
	\sum_{\substack{\nu_0+\cdots+\nu_{k-1}<0 \\ \nu_0+\cdots+\nu_k\ge0}}.
	$$
	Note here that, if $\nu_0+\cdots+\nu_k\ge0$, then $\nu_0+\cdots+\nu_l\ge0$ for any $l\ge k$,
	since $\nu_1,\cdots,\nu_r\ge1$.
		
	So we have
	\begin{align*}
		\sum_{\substack{\nu_0+\cdots+\nu_{k-1}<0 \\ \nu_0+\cdots+\nu_k\ge0}}
		&\chi(-1)^{\sum_{l=0}^{k-1}\nu_l(n^{(l)}-n^{(k)})}
		q^{\sum_{l=0}^{k-1}\nu_l(sn^{(l)}-sn^{(k)}-n(l))}
		\prod_{l=0}^r\prod_{\substack{i\in I_l \\ \sigma(i)=i}}q'_{i, \lambda_l}(\sigma; B) \\
		&=\sum_{\nu_0, \cdots, \nu_{k-1}}\chi(-1)^{\sum_{l=0}^{k-1}\nu_l(n^{(l)}-n^{(k)})}
		q^{\sum_{l=0}^{k-1}\nu_l(sn^{(l)}-sn^{(k)}-n(l))}
		\prod_{l=0}^{k-1}\prod_{\substack{i\in I_l \\ \sigma(i)=i}}q'_{i, \lambda_l}(\sigma; B) \\
		&\quad\times\sum_{\nu_k=-(\nu_0+\cdots+\nu_{k-1})}^\infty
		\sum_{\nu_{k+1}=1}^\infty\cdots\sum_{\nu_r=1}^\infty
		q^{-\sum_{l=k}^r\nu_ln(l)}\{2(1-q^{-1})\}^{\sum_{l=k}^rc_1^{(l)}(\sigma)} \\
		&=\frac{\{2(1-q^{-1})\}^{\sum_{l=k}^rc_1^{(l)}(\sigma)}q^{n(k)}}
		{\prod_{l=k}^r(q^{n(l)}-1)} \\
		&\quad\times\sum_{\nu_0,\cdots,\nu_{k-1}}\chi(-1)^{\sum_{l=0}^{k-1}\nu_l(n^{(l)}-n^{(k)})}
		q^{\sum_{l=0}^{k-1}\nu_l((sn^{(l)}-n(l))-(sn^{(k)}-n(k)))}
		\prod_{l=0}^{k-1}\prod_{\substack{i\in I_l \\ \sigma(i)=i}}q'_{i, \lambda_l}(\sigma; B).\tag{5}
	\end{align*}
	It is easy to be checked that,
	when $\lambda_l\ge0$,
	$$\prod_{\substack{i\in I_l \\ \sigma(i)=i}}q'_{i, \lambda_l}(\sigma; B)=\{2(1-q^{-1})\}^{c_1^{(l)}(\sigma)}$$
	where we put
	$$c_1^{(k)}(\sigma)=\#\{i\in I_l \mid \sigma(i)=i\}.$$
	
	\begin{lemma} \label{q}
		When we put
		$$r_i'=\frac{1}{2}\sum_{k=1}^{i-1}\min\{e_k+\lambda_l, 0\}
		+\frac{1}{2}\sum_{k=i+1}^n\min\{e_k+\lambda_l+2, 0\},$$
		we have
		$$q_{i, \lambda_l}(\sigma, B)
		=2\alpha_\psi(\pi)q^{r_i'+\frac{1}{2}\min\{e_i+\lambda_l, 0\}}\xi_{i, \lambda_l}(B)_\chi.$$
	\end{lemma}
	
	\begin{proof}
		
		Now we use the fact that, for $v\in\mathfrak{o}^\times$ and $k\in\mathbb{Z}$,
		$$I(v\pi^k)=
		\begin{cases}
			1 & k\ge0 \\
			\alpha_\psi(v\pi^k)q^{k/2} & k<0,
		\end{cases}
		$$
		so it becomes
		$$q_{i, \lambda_l}(\sigma; B)
		=q^{r_i'}\sum_{\varepsilon=1, \delta}\left\{\chi(\varepsilon)
		I^*(-\varepsilon v_i\pi^{e_i+\lambda_l})
		\prod_{\substack{k=1 \\ e_k+\lambda_l<0}}^{i-1}
		\alpha_\psi(-\varepsilon v_k\pi^{e_k+\lambda_l})
		\prod_{\substack{k=i+1 \\ e_k+\lambda_l+2<0}}^n
		\alpha_\psi(-\varepsilon v_k\pi^{e_k+\lambda_l+2})\right\}.$$
		
		We recall
		$$I^*(v\pi^k)=
		I(v\pi^k)-\frac{1}{q}I(v\pi^{k+2})=
		\begin{cases}
			1-q^{-1} & (k\ge0) \\
			\alpha_\psi(v\pi^{-1})q^{-1/2}-q^{-1} & (k=-1) \\
			0 & (k\le-2).
		\end{cases}$$
		Hence
		\begin{align*}
			q_{i, \lambda_l}(\sigma; B)
			&=q^{r_i'}\sum_{\varepsilon=1, \delta}\left\{\chi(\varepsilon)
			\prod_{\substack{k=1 \\ e_k+\lambda_l<0}}^{i-1}
			\alpha_\psi(-\varepsilon v_k\pi^{e_k+\lambda_l})
			\prod_{\substack{k=i+1 \\ e_k+\lambda_l+2<0}}^n
			\alpha_\psi(-\varepsilon v_k\pi^{e_k+\lambda_l+2})\right\} \\
			&\quad\times
			\begin{cases}
				1-q^{-1} & (e_i+\lambda_l\ge0) \\
				\alpha_\psi(-\varepsilon v_i\pi^{-1})q^{-1/2}-q^{-1} & (e_i+\lambda_l=-1). \\
			\end{cases}
		\end{align*}
		
		To calculate $q_{i, \lambda_l}(\sigma; B)$, we use following lemmas.
		\begin{lemma}\label{chisum}
			For $a_k\in F^\times(1\le k\le r)$, we define
			$A:=\{ k \mid \mathrm{ord}\,a_k:\mathrm{odd}\}$. In this situation, 
			$$
			\sum_{\varepsilon=1, \delta}
			\chi(\varepsilon)\prod_{k=1}^r\alpha_\psi(\varepsilon a_k)
			=
			\begin{cases}
				\displaystyle2\prod_{k=1}^r\alpha_\psi(a_k) & \#A:\mathrm{odd} \\
				0 & \#A:\mathrm{even}.
			\end{cases}$$
		\end{lemma}
		
		\begin{proof}
			From Lemma \ref{Weildelta}, we have $\displaystyle \alpha_\psi(\delta a)=(-1)^{\mathrm{ord} a}\alpha_\psi(a)$.
			Hence
			\begin{align*}
				\sum_{\varepsilon=1, \delta}\chi(\varepsilon)\prod_{k=1}^r\alpha_\psi(\varepsilon a_k)
				&=\prod_{k=1}^r\alpha_\psi(a_k)+\chi(\delta)\prod_{k=1}^r\alpha_\psi(\delta a_k) \\
				&=\prod_{k=1}^r\alpha_\psi(a_k)-\prod_{k=1}^r(-1)^{\mathrm{ord} a_k}\alpha_\psi(a_k) \\
				&=\left(1-\prod_{k=1}^r(-1)^{\mathrm{ord} a_k}\right)\prod_{k=1}^r\alpha_\psi(a_k).
			\end{align*}
			Since the term $1-\prod_{k=1}^r(-1)^{\mathrm{ord} a_k}$ is 2 and 0 when $\#A$ is odd and even, respectively,
			we can prove this lemma.
		\end{proof}
		
		\begin{lemma}\label{chi}
			Let $a, b\in\pi\mathfrak{o}^\times$. Then
			$\alpha_\psi(a)\alpha_\psi(b)=\chi(-ab)$.
		\end{lemma}
		\begin{proof}
			From the Lemma \ref{Hilbert},
			\begin{align*}
				\frac{\alpha_\psi(a)\alpha_\psi(b)}{\alpha_\psi(1)\alpha_\psi(ab)}=\langle a, b\rangle.
			\end{align*}
			Now the denominator of the left side $\alpha_\psi(1)\alpha_\psi(ab)=1\cdot1=1$ and the
			right side of the Hilbert symbol is $\chi(-ab)$, so the assertion is proved.
		\end{proof}
		First, we assume $e_i+\lambda_l\ge0$.
		By Lemma \ref{chisum}, we may assume that the order of the set
		\begin{align*}
			B_i(\lambda_l)
			&=\{ k \mid 1\le k\le i-1,\,\, e_k+\lambda_l<0,\,\, e_k\not\equiv\lambda_l\bmod2\} \\
			&\quad\cup \{ k \mid i+1\le k\le n,\,\, e_k+\lambda_l+2<0,\,\, e_k\not\equiv\lambda_l\bmod2\}
		\end{align*}
		is odd. If we fix an element $k_0\in B_i(\lambda_l)$,
		then $\#B_i(\lambda_l)\backslash\{k_0\}$ is even and by Lemma \ref{chi}, we give
		
		\begin{align*}
			q_{i, \lambda_l}(\sigma; B)
			&=(1-q^{-1})q^{r_i'}\sum_{\varepsilon=1, \delta}\left\{\chi(\varepsilon)
			\prod_{\substack{k=1 \\ e_k+\lambda_l<0}}^{i-1}
			\alpha_\psi(-\varepsilon v_k\pi^{e_k+\lambda_l})
			\prod_{\substack{k=i+1 \\ e_k+\lambda_l+2<0}}^n
			\alpha_\psi(-\varepsilon v_k\pi^{e_k+\lambda_l+2})\right\} \\
			&=2(1-q^{-1})q^{r_i'}\prod_{k\in B_i(\lambda_l)}
			\alpha_\psi(-v_k\pi^{e_k+\lambda_l}) \\
			&=2(1-q^{-1})q^{r_i'}\prod_{k\in B_i(\lambda_l)\backslash\{k_0\}}
			\alpha_\psi(-v_k\pi^{e_k+\lambda_l})\cdot\alpha_\psi(-v_{k_0}\pi^{e_{k_0}+\lambda_l}) \\
			&=2(1-q^{-1})q^{r_i'}\chi(-1)^{[\#B_i(\lambda_l)/2]}\prod_{k\in B_i(\lambda_l)\backslash\{k_0\}}
			\chi(v_k)\alpha_\psi(-v_{k_0}\pi^{e_{k_0}+\lambda_l}) \\
			&=2(1-q^{-1})q^{r_i'}\chi(-1)^{[\#B_i(\lambda_l)/2]+1}\alpha_\psi(\pi)\prod_{k\in B_i(\lambda_l)}\chi(v_k).
		\end{align*}
		We use $\alpha_\psi(-v_{k_0}\pi^{e_{k_0}+\lambda_l})=\chi(-v_{k_0})\alpha_\psi(\pi)$,
		since $e_{k_0}+\lambda_l$ is an odd number.
		
		When $e_i+\lambda_l=-1$, we can similarly calculate as
		
		\begin{align*}
			&q_{i, \lambda_l}(\sigma; B) \\
			&=\begin{cases}
				\displaystyle
				2q^{r_i'-\frac{1}{2}}\chi(v_i)\chi(-1)^{[\#B_i(\lambda_l)/2]+1}\alpha_\psi(\pi)\prod_{k\in B_i(\lambda_l)}\chi(v_k)
				& \#B_i(\lambda_l): \text{even} \\
				\displaystyle
				-2q^{r_i'-1}\chi(-1)^{[\#B_i(\lambda_l)/2]+1}\alpha_\psi(\pi)\prod_{k\in B_i(\lambda_l)}\chi(v_k)
				& \#B_i(\lambda_l): \text{odd}. \\
			\end{cases}
		\end{align*}
		
		Therefore,
		$$q_{i, \lambda_l}(\sigma, B)
		=2\alpha_\psi(\pi)q^{r_i'+\frac{1}{2}\min\{e_i+\lambda_l, 0\}}\xi_{i, \lambda_l}(B)_\chi$$
		and the lemma holds.
	\end{proof}
	
	From (4), (5) and Lemma 4.1, It follows that
	\begin{align*}
		S_t(B, s)^\chi
		&=\sum_{\substack{\sigma\in\mathfrak{S}_n \\ \sigma^2=1}}
		\chi(-1)^{c_2^t(\sigma)}2^{-c_1(\sigma)}(1-q^{-1})^{c_2(\sigma)}q^{-c_2(\sigma)}
		\sum_{\substack{I=I_0\cup\cdots\cup I_r \\ n^{(k)}=t}}q^{-\tau(\{I_i\})-t(\sigma, \{I_i\})} \\
		&\quad\times\frac{\{2(1-q^{-1})\}^{\sum_{l=k}^rc_1^{(l)}(\sigma)}q^{n(k)}}
		{\prod_{l=k}^r(q^{n(l)}-1)} \\
		&\quad\times\sum_{\nu_0,\cdots,\nu_{k-1}}\chi(-1)^{\sum_{l=0}^{k-1}\nu_l(n^{(l)}-n^{(k)})}
		q^{d(\sigma, e, \beta)+\sum_{l=0}^{k-1}\nu_l((sn^{(l)}-n(l))-(sn^{(k)}-n(k)))} \\
		&\quad\times\prod_{l=0}^{k-1}\prod_{\substack{i\in I_l \\ \sigma(i)=i}}
		2\alpha_\psi(\pi)q^{r_i'+\frac{1}{2}\min\{\beta_i+\lambda_l, 0\}}\xi_{i, \lambda_l}(B)_\chi,
	\end{align*}
	so we have
	\begin{align*}
		S_t(B, s)^\chi
		&=\alpha_\psi(\pi)^{n-t}\sum_{\substack{\sigma\in\mathfrak{S}_n \\ \sigma^2=1}}
		(1-q^{-1})^{c_2(\sigma)}q^{-c_2(\sigma)}
		\sum_{\substack{I=I_0\cup\cdots\cup I_r \\ n^{(k)}=t}}q^{-\tau(\{I_i\})-t(\sigma, \{I_i\})}\frac{
			(1-q^{-1})^{\sum_{l=k}^rc_1^{(l)}(\sigma)}q^{n(k)}}
		{\prod_{l=k}^r(q^{n(l)}-1)} \\
		&\quad\times\sum_{\{\nu\}_k^t}\prod_{l=0}^{k-1}\chi(-1)^{\nu_l(n^{(l)}-n^{(k)})}
		q^{\nu_l((sn^{(l)}-n(l))-(sn^{(k)}-n(k)))+\tilde{\rho}_{l, \nu_0+\cdots+\nu_l}(\sigma; B)}
		\prod_{\substack{i\in I_l \\ \sigma(i)=i}}\xi_{i, \nu_0+\cdots+\nu_l}(B)_\chi
	\end{align*}
	To show the main theorem, we recall
	\begin{align*}
		&\tilde{d}(\sigma, \lambda, e)+\sum_{l=0}^{k-1}\sum_{\substack{i\in I_l \\ \sigma(i)=i}}\left(
		r_i'+\frac{1}{2}\min\{e_i+\lambda_l, 0\}\right) \\
		&=\frac{1}{2}\sum_{l=k}^r\sum_{\substack{i\in I_l \\ \sigma(i)\neq i}}
		\sum_{u=1}^n\min\{e_u+\lambda_l+e_{\sigma, i, u}, 0\}
		+\frac{1}{2}\sum_{l=0}^{k-1}\sum_{i\in I_l}
		\sum_{u=1}^n\min\{e_u+\lambda_l+e_{\sigma, i, u}, 0\} \\
		&=\sum_{l=0}^{k-1}\tilde{\rho}_{l, \lambda_l}(\sigma; B).
	\end{align*}
\end{proof}

\subsection{The case $\varphi=f_n$} 

From the definition of the Siegel series, we have
$$S_n(B, s)=\int_{\mathrm{Sym}_n(F)}f_n\left(w_n\begin{pmatrix}
	1 & X \\ 0 &1
\end{pmatrix}\right)\psi(-\mathrm{tr}(BX))dX=1.$$
On the other hand, from the main theorem we get
\begin{align*}
	S_n(B, s)^\chi
	=\sum_{\substack{\sigma\in\mathfrak{S}_n \\ \sigma^2=1}}
	(1-q^{-1})^{n-c_2(\sigma)}q^{-c_2(\sigma)}
	\sum_{I=I_0\cup\cdots\cup I_r}q^{-\tau(\{I_i\})-t(\sigma, \{I_i\})}
	\times\frac{q^{\frac{n(n+1)}{2}}}{\prod_{l=0}^r(q^{n(l)}-1)}.
\end{align*}

It follows that there is a nontrivial equation: 
\begin{proposition}
	When we use the notation as above, we have 
	\begin{align*}
		\sum_{\substack{\sigma\in\mathfrak{S}_n \\ \sigma^2=1}}
		(1-q^{-1})^{n-c_2(\sigma)}q^{-c_2(\sigma)}
		\sum_{I=I_0\cup\cdots\cup I_r}q^{-\tau(\{I_i\})-t(\sigma, \{I_i\})}
		\times\frac{q^{\frac{n(n+1)}{2}}}{\prod_{l=0}^r(q^{n(l)}-1)}=1.
	\end{align*}
\end{proposition}

\subsection{The case $\omega=\mathbf{1}$}
\begin{theorem}
	When we define
	\begin{align*}
		\xi_{i, \lambda}(B)_1
		&=\prod_{k \in B_i(\lambda)}\chi(v_k) \times
		\begin{cases}
			0 & e_i+\lambda\ge0,\,\, \#B_i(\lambda) : \text{odd} \\
			(1-q^{-1})\chi(-1)^{[\#B_i(\lambda)/2]} & e_i+\lambda\ge0,\,\, \#B_i(\lambda) : \text{even} \\
			\chi(v_i)\chi(-1)^{[\#B_i(\lambda)/2]+1} & e_i+\lambda=-1,\,\, \#B_i(\lambda) : \text{odd} \\
			-q^{-1/2}\chi(-1)^{[\#B_i(\lambda)/2]} & e_i+\lambda=-1,\,\, \#B_i(\lambda) : \text{even},
		\end{cases}
	\end{align*}
	the Siegel series
	\begin{align*}
		S_0(B, s)^{\mathbf{1}}&=\int_{\mathrm{Sym}_n(F)}f_0\left(w_n
		\begin{pmatrix}
			1 & X \\ 0 & 1
		\end{pmatrix}
		\right)\psi(-\mathrm{tr}(BX))dX
	\end{align*}
	where $f_0 \in I_n\left(\mathbf{1}, s-\frac{n+1}{2}\right)^{\Gamma, \mathbf{1}}$
	is equal to
	\begin{align*}
		S_0(B, s)^{\mathbf{1}}
		&=\sum_{\substack{\sigma\in\mathfrak{S}_n \\ \sigma^2=1}}
		(1-q^{-1})^{c_2(\sigma)}
		\sum_{I=I_0\cup\cdots\cup I_r}q^{-c_2(\sigma)-\tau(\{I_i\})-t(\sigma, \{I_i\})} \\
		&\quad\times\sum_{\{\nu\}}\prod_{l=0}^r
		q^{\frac{1}{2}\nu_ln^{(l)}(2s-1-n^{(l)})+\tilde{\rho}_{l, \nu_0+\cdots+\nu_l}(\sigma; B)}
		\prod_{\substack{i\in I_l \\ \sigma(i)=i}}\xi_{i, \nu_0+\cdots+\nu_l}(B)_1.
	\end{align*}
	
	The index $\{\nu\}$ runs through the finite set
	$$
	\{(\nu_0, \nu_1, \cdots, \nu_r)\in\mathbb{Z}\times\mathbb{Z}_{>0}^r\mid
	-b_l(\sigma, B)\le \nu_0+\nu_1+\cdots+\nu_l\le-1, (0\le l\le r)\}.$$
\end{theorem}
\begin{proof}
	An explicit formula arises similar to the proof of theorem \ref{mainthm} except for $\xi_{i, \lambda}(B)_1$.
	We check this term.
	
	Since
	\begin{align*}
		&\sum_{\varepsilon=1, \delta}\left\{
		I^*(-\varepsilon v_i\pi^{e_i+\lambda_l})
		\prod_{\substack{k=1 \\ e_k+\lambda_l<0}}^{i-1}
		\alpha_\psi(-\varepsilon v_k\pi^{e_k+\lambda_l})
		\prod_{\substack{k=i+1 \\ e_k+\lambda_l+2<0}}^n
		\alpha_\psi(-\varepsilon v_k\pi^{e_k+\lambda_l+2})\right\} \\
		&=\sum_{\varepsilon=1, \delta}\left\{
		\prod_{\substack{k=1 \\ e_k+\lambda_l<0}}^{i-1}
		\alpha_\psi(-\varepsilon v_k\pi^{e_k+\lambda_l})
		\prod_{\substack{k=i+1 \\ e_k+\lambda_l+2<0}}^n
		\alpha_\psi(-\varepsilon v_k\pi^{e_k+\lambda_l+2})\right\} \\
		&\quad\times
		\begin{cases}
			1-q^{-1} & (e_i+\lambda_l\ge0) \\
			\alpha_\psi(-\varepsilon v_i\pi^{-1})q^{-1/2}-q^{-1} & (e_i+\lambda_l=-1)
		\end{cases} \\
		&=2q^{\frac{1}{2}\min\{\beta_i+\lambda_l, 0\}}\prod_{k \in B_i(\lambda)}\chi(v_k) \\
		&\quad\times
		\begin{cases}
			0 & e_i+\lambda_l\ge0,\,\, \#B_i(\lambda_l) : \text{odd} \\
			(1-q^{-1})\chi(-1)^{[\#B_i(\lambda)/2]} & e_i+\lambda_l\ge0,\,\, \#B_i(\lambda_l) : \text{even} \\
			\chi(v_i)\chi(-1)^{[\#B_i(\lambda)/2]+1} & e_i+\lambda_l=-1,\,\, \#B_i(\lambda_l) : \text{odd} \\
			-q^{-1/2}\chi(-1)^{[\#B_i(\lambda)/2]} & e_i+\lambda_l=-1,\,\, \#B_i(\lambda_l) : \text{even}
		\end{cases}
	\end{align*}
	the statement of the theorem holds.
\end{proof}

\section{Values of the Siegel series for $n=1, 2, 3$}
\subsection{The case $n=1$}
When $n=1$, $\sigma=\mathrm{id}$, the identity permutation, and $I_0=\{1\}$.
Thus we obtain the following theorem.
\begin{theorem}
	Let $B=(\alpha_1\pi^{e_1})\in\mathrm{Sym}_1(F)$ $(\alpha_1\in\mathfrak{o}^\times, e_1\ge0)$.
	Then the Siegel series $S_0(B, s)^\chi$ is
	$$S_0(B, s)^\chi=\alpha_\psi(\pi)\chi(-1)^{e_1}\chi(\alpha_1)q^{(1-s)(e_1+1)-\frac{1}{2}}.$$
\end{theorem}
\begin{proof}
	When $n=1$, $\sigma=\text{id}$ and the partition $\{I_i\}$ is $I_0=\{1\}$, so
	\begin{align*}
		S_0(B, s)^\chi&=\alpha_\psi(\pi)\sum_{\nu_0=-e_1-1}^{-1}
		\chi(-1)^{\nu_0}q^{\frac{1}{2}\nu_0(2s-1-1)+\tilde{\rho}_{0, \nu_0}(\sigma; B)}\xi_{1, \nu_0}(B)_\chi \\
		&=\alpha_\psi(\pi)\chi(-1)^{-e_1-1}q^{(1-s)(e_1+1)+\tilde{\rho}_{0, -e_1-1}(\sigma; B)}\xi_{1, -e_1-1}(B)_\chi,
	\end{align*}
	since $B_1(\nu_0)=\emptyset$ so $\xi_{1,\nu_0}(B)_\chi=0$ ($\nu_0\ge-e_1$).
	We follow this theorem by $\tilde{\rho}_{0, -e_1-1}(\sigma; B)=-\frac{1}{2}$ and  $\xi_{1, -e_1-1}(B)_\chi=\chi(-\alpha_1)$.
\end{proof}

\subsection{The case $n=2$}

We calculate the explicit value for the Siegel series $S(B, s)$ when 
$B=\mathrm{diag}(\alpha_1\pi^{e_1}, \alpha_2\pi^{e_2})$
($0\le e_1\le e_2,\, \alpha_1, \alpha_2\in\mathfrak{o}^\times$).

\begin{theorem}[\cite{degree2} Proposition 3.1] \label{degree2}
	In above situation, the Siegel series is calculated as follows.
	\begin{itemize}
		\item $e_1+e_2$ is even
		\begin{align*}
			S_0(B, s)^\chi&=\chi(-1)\left\{
			(q-1)\sum_{i=1}^{\frac{1}{2}(e_1+e_2)}q^{(3-2s)i-2}-q^{(3-2s)(\frac{e_1+e_2}{2}+1)-2}\right\} \\
			S_1(B, s)^\chi&=\alpha_\psi(\pi)\chi(-1)^{e_1+1}\chi(-\alpha)q^{(\frac{3}{2}-s)(e_1+1)+\frac{e_1}{2}}
			\frac{(1-q^{2-2s})(1-q^{(\frac{3}{2}-s)(e_2-e_1)})}{1-q^{3-2s}}
		\end{align*}
		\item $e_1+e_2$ is odd
		\begin{align*}
			S_0(B, s)^\chi&=\chi(-1)\left\{
			(q-1)\sum_{i=1}^{\frac{1}{2}(e_1+e_2+1)}q^{(3-2s)i-2}+\chi(\pi\alpha_1\alpha_2)q^{(3-2s)(\frac{e_1+e_2}{2}+1)-\frac{3}{2}}\right\} \\
			S_1(B, s)^\chi&=\alpha_\psi(\pi)\chi(-1)^{e_1+1}\chi(-\alpha)q^{(\frac{3}{2}-s)(e_1+1)+\frac{e_1}{2}}
			\frac{1+\chi(-\alpha\beta)q^{1-s}}{1-q^{3-2s}} \\
			&\quad\times\left\{1-\chi(-\alpha\beta)q^{1-s}+\chi(-\alpha\beta)q^{(\frac{3}{2}-s)(e_2-e_1)-\frac{1}{2}}
			-q^{(\frac{3}{2}-s)(e_2-e_1)}\right\}
		\end{align*}
	\end{itemize}
\end{theorem}
This statement holds since we calculate the sum of 4 cases below as well as when $n=1$.
\begin{enumerate}[(1)]
	\setlength{\leftskip}{2.0cm}
	\item $\sigma=\text{id}$, $I_0=\{1\}, I_1=\{2\}$
	\item $\sigma=\text{id}$, $I_0=\{1, 2\}$
	\item $\sigma=\text{id}$, $I_0=\{2\}, I_1=\{1\}$
	\item $\sigma=(1 2)$, $I_0=\{1, 2\}.$
\end{enumerate}

\subsection{The case $n=3$}

In this section, we give the explicit value of the Siegel series for $n=3$.
We put $B=\mathrm{diag}(\alpha_1\pi^{e_1}, \alpha_2\pi^{e_2}, \alpha_3\pi^{e_3})$,
where $\alpha_i\in\mathfrak{o}^\times$, and $e_1\le e_2\le e_3$.
In this case, we can calculate the value of the Siegel series in the same way where $n$ is 1 and 2.

\begin{theorem}
	For $(0\le d\le 3)$, the value of the Siegel series $S_{3, d}(B, s)^\chi$ is calculated as follows.
	Here we note $\mathrm{Wh}_{3, d}(B, s)^\chi = S_{3, d}(B, s+2)^\chi$.
	\begin{itemize}
		\item $d=0$ and $e_1+e_2$ is even
		\begin{align*}
			\mathrm{Wh}_{3, 0}(B, s)^\chi
			&=\alpha_\psi(\pi)^3\chi(-1)^{e_1}q^{-(e_1+3)s-\frac{5}{2}}(1-q^{-2s-1})^{-1}(1-q^{-2s+1})^{-1} \\
			&\quad\times\{\chi(\alpha_1)q^{e_1+1}(1-q^{-1})(1-q^{-2s-1})(1-q^{-(e_2-e_1)s+\frac{e_2-e_1}{2}}) \\
			&\qquad+\chi(-\alpha_3)\chi(\alpha_1\alpha_2)^{e_2+e_3+1}q^{-(e_3-e_1)s}(1-q^{-2s+1}) \\
			&\quad\qquad\times(q^{\frac{e_1+e_2}{2}}(1-q^{-1})-q^{-(e_1+e_2)s}(1-q^{-2s-2}))\}.
		\end{align*}
		\item $d=0$ and $e_1+e_2$ is odd
		\begin{align*}
			\mathrm{Wh}_{3, 0}(B, s)^\chi
			&=\alpha_\psi(\pi)^3\chi(-1)^{e_1}q^{-(e_1+3)s-\frac{5}{2}}(1-q^{-2s-1})^{-1}(1-q^{-2s})^{-1}(1-q^{-2s+1})^{-1} \\
			&\quad\times\{\chi(-\alpha_2)q^{-(e_2-e_1)s+\frac{e_1+e_2+1}{2}}(1-q^{-1})(1-q^{-2s-1})(1-q^{-2s+1}) \\
			&\qquad\qquad\times(1-\chi(-\alpha_1\alpha_2)^{e_2+e_3}q^{-s(e_3-e_2)}) \\
			&\quad\quad-\chi(\alpha_1)q^{-(e_2-e_1+1)s+\frac{e_1+e_2+1}{2}}(1-q^{-1})^2 \\
			&\qquad\qquad\times(q(1-q^{-2s-1})+\chi(-\alpha_1\alpha_2)^{e_2+e_3}q^{-s(e_3-e_2)}(1-q^{-2s+1})) \\
			&\quad\quad+\chi(\alpha_1)q^{e_1+1}(1-q^{-1})(1-q^{-2s})(1-q^{-2s-1}) \\
			&\quad\quad-\chi(\alpha_1)\chi(-\alpha_1\alpha_2)^{e_2+e_3}q^{-(e_2+e_3)s}(1-q^{-2s})(1-q^{-2s-2})(1-q^{-2s+1})\}.
		\end{align*}
		\item $d=1$ and $e_1+e_2$ is even
		\begin{align*}
			\mathrm{Wh}_{3,1}(B, s)^\chi
			&=q^{-2s-1}(1-q^{-2s-1})^{-1}(1-q^{-2s+1})^{-1} \\
			&\quad\times\{\chi(-1)q(1-q^{-3})(1-q^{-2s-1}) \\
			&\quad\quad-\chi(-1)q^{-(e_1+e_2)s+\frac{e_1+e_2}{2}+1}(1-q^{-2s-1})(1-q^{-2s-2}) \\
			&\quad\quad+\chi(\alpha_2\alpha_3)\chi(\alpha_1\alpha_2)^{e_2+e_3}q^{-(e_1+e_3)s+\frac{e_1+e_2}{2}} \\
			&\qquad\qquad\times(1-q^{-2s+1})(1-q^{-2s-2})(1-q^{-s(e_2-e_1)-\frac{e_2-e_1}{2}})\}.
		\end{align*}
		\item $d=1$ and $e_1+e_2$ is odd
		\begin{align*}
			\mathrm{Wh}_{3,1}(B, s)^\chi
			&=\chi(-1)q^{-2s-1}(1-q^{-2s-1})^{-1}(1-q^{-2s})^{-1}(1-q^{-2s+1})^{-1} \\
			&\quad\times\{q(1-q^{-3})(1-q^{-2s})(1-q^{-2s-1}) \\
			&\quad\quad-q^{-(e_1+e_2+1)s+\frac{e_1+e_2+3}{2}}(1-q^{-2})(1-q^{-2s})(1-q^{-2s-2}) \\
			&\quad\quad-\chi(-\alpha_1\alpha_2)^{e_2+e_3}q^{-(e_2+e_3)s+e_1}(1-q^{-2s})(1-q^{-2s+1})(1-q^{-2s-2}) \\
			&\quad\quad+(1-\chi(-\alpha_1\alpha_2)^{e_2+e_3}q^{-s(e_3-e_2)})q^{-(e_1+e_2+1)s+\frac{e_1+e_2+1}{2}} \\
			&\quad\qquad\times(1-q^{-1})(1-q^{-2s+1})(1-q^{-2s-2}) \\
			&\quad\quad+(1-\chi(-\alpha_1\alpha_2)^{e_2+e_3}q^{-s(e_3-e_2)})\chi(-\alpha_1\alpha_2)q^{-(e_1+e_2)s+\frac{e_1+e_2+1}{2}} \\
			&\quad\qquad\times(1-q^{-2s-2})(1-q^{-2s-1})(1-q^{-2s+1})\}.
		\end{align*}
		\item $d=2$ and $e_1+e_2$ is even
		\begin{align*}
			\mathrm{Wh}_{3,2}(B, s)^\chi
			&=\alpha_\psi(\pi)\chi(-1)^{e_1+1}q^{-(e_1+1)s+e_1-\frac{1}{2}}(1-q^{-2s+1})^{-1} \\
			&\quad\times\{\chi(-\alpha_1)q(1-q^{-2s-1})(1-q^{-s(e_2-e_1)+\frac{e_2-e_1}{2}}) \\
			&\quad\quad+\chi(\alpha_1\alpha_2)^{e_2+e_3+1}\chi(\alpha_3)q^{-(e_3-e_1)s+\frac{e_2-e_1}{2}}(1-q^{-2s+1})\}.
		\end{align*}
		\item $d=2$ and $e_1+e_2$ is odd
		\begin{align*}
			\mathrm{Wh}_{3,2}(B, s)^\chi
			&=\alpha_\psi(\pi)\chi(-1)^{e_1}q^{-(e_1+1)s+e_1}(1-q^{-2s})^{-1}(1-q^{-2s+1})^{-1} \\
			&\quad\times\{\chi(-\alpha_2)q^{-s(e_2-e_1)+\frac{e_2-e_1}{2}}(1-q^{-2s-1})(1-q^{-2s+1}) \\
			&\quad\qquad\times(1-\chi(-\alpha_1\alpha_2)^{e_2+e_3}q^{-s(e_3-e_2)}) \\
			&\quad\quad+\chi(\alpha_1)q^{\frac{1}{2}}(1-q^{-2s-1})((1-q^{-2s})-q^{-(e_2-e_1+1)s+\frac{e_2-e_1+1}{2}}(1-q^{-1})) \\
			&\quad\quad-\chi(\alpha_1)\chi(-\alpha_1\alpha_2)^{e_2+e_3}q^{-(e_3-e_1+1)s+\frac{e_2-e_1}{2}}(1-q^{-1})(1-q^{-2s+1})\}.
		\end{align*}
	\end{itemize}
\end{theorem}

\section*{Statements and Declarations}
\paragraph{Funding}
The authors did not receive support from any organization for the submitted work.
\paragraph{Competing Interests}
The authors have no competing interests to declare.
\paragraph{Data Availability}
Data sharing not applicable to this article as no datasets were generated or analyzed.


\begin{thebibliography}{99}
	\bibitem{padicbook} W. Casselman: Introduction to the theory of admissible representations of $p$-adic reductive groups, preprint (1995).
	\bibitem{Casselman} W. Casselman: The unramified principal series of $p$-adic groups.
	I. The spherical function,
	Compositio Mathematica, tome 40, no 3 (1980), 387-406.
	\bibitem{CS} W. Casselman, J. Shalika: The unramified principal series of $p$-adic groups.
	II. The Whittaker function,
	Compositio Mathematica, tome 41, no 2 (1980), 207-231.
	\bibitem{Cassels} J. W. S. Cassels: Rational quadratic forms, London Mathematical Society Monographs No. 13,
	Academic Press, London-New York-San Francisco, 1978.
	\bibitem{GJ} R. Godement, H. Jacquet: Zeta Functions of Simple Algebras,
	Lecture Notes in Mathematics, Vol. 260, Springer-Verlag, Berlin-New York (1972).
	\bibitem{Gunj} K. Gunji: On the Siegel Eisenstein series of degree two for low weights, J. Math. Soc. Japan 67 (2015), no. 3, 1043-1067. https://doi.org/10.2969/jmsj/06731043
	\bibitem{Gunji} K. Gunji: On the computation of ramified Siegel series of degree 3,
	RIMS Kokyuroku, No. 2100 (2019), 165-178.
	\bibitem{Gunji2} K. Gunji: On the Fourier coefficients of the Siegel Eisenstein series of odd level and the genus theta series, Journal of Number Theory 240 (2022), 124-144. https://doi.org/10.1016/j.jnt.2022.01.004
	\bibitem{SatoHironaka} Y. Hironaka, F. Sato: Local densities of representations of quadratic forms over
	$p$-adic integers (the non-dyadic case), J. Number Theory 83 (2000), 106-136. https://doi.org/10.1006/jnth.1999.2505
	\bibitem{FEIkeda} T. Ikeda: On the functional equation of the Siegel series,
	Journal of Number Theory 172 (2017), 44-62. https://doi.org/10.1016/j.jnt.2016.08.002
	\bibitem{Location} T. Ikeda: On the location of the poles of the triple $L$-functions, 
	Composito Mathematica 83.2 (1992), 187-237.
	\bibitem{Katsurada} H. Katsurada: An explicit formula for Siegel series, Amer. J. Math.121 (1999), 415-452.
	\bibitem{Mizuno} Y. Mizuno: An explicit arithmetic formula for the Fourier coefficients of Siegel-Eisenstein series of degree two and square-free odd levels, Math. Z. 263 (2009), no. 4, 837-860. https://doi.org/10.1007/s00209-008-0442-2
	\bibitem{Moeglin} C. M\oe glin et J.L. Waldspurger: Mod\`{e}les de Whittaker d\'{e}g\'{e}n\'{e}r\'{e}s
	pour des groupes $p$-adiques, Math. Z. 196 (1987), 427-452.
	\bibitem{Shimura} G. Shimura: Euler products and Eisenstein series, CBMS Regional Conference
	Series in Mathematics, 93 (1997), AMS.
	\bibitem{Sweet} W. Jay Sweet Jr: A Computation of the Gamma Matrix of a Family of $p$-adic Zeta Integrals, Journal of Number Theory, 55 (1995), 222-260.
	\bibitem{degree2} S. Takemori: $p$-adic Siegel-Eisenstein series of degree two,
	Journal of Number Theory 132 (2012), 1203-1264. https://doi.org/10.1016/j.jnt.2012.01.001
	\bibitem{Takemori} S. Takemori: Siegel Eisenstein series of degree $n$ and $\Lambda$-adic
	Eisenstein series, Journal of Number Theory 149 (2015), 105-138. https://doi.org/10.1016/j.jnt.2014.10.005
	\bibitem{NTG} J. Tate: Number Theoretic Background,
	Proceedings of Symposia in Pure Mathematics Vol. 33 (1979), part 2, 3-26.
	\bibitem{O'Meara} O. Timothy O'Meara: Introduction to Quadratic Forms, Springer-Verlag, New York, 1963.
	\bibitem{BNT} A. Weil: Basic Number Theory, Springer-Verlag, New York, 1967.
\end{thebibliography}
\end{document}